\documentclass[10pt, reqno]{amsart}
\usepackage{amsmath}
\usepackage{amsthm, color}
\usepackage{amsfonts}
\usepackage{amssymb}
\usepackage{mathrsfs}
\usepackage[margin=.95in]{geometry}
\usepackage{float}
\usepackage{set space}
\usepackage[dvipsnames]{xcolor}
\usepackage{tikz} 
\usetikzlibrary{matrix, arrows, quotes, calc, intersections, automata, positioning, shapes.geometric}
\usepackage{tikz-cd}
\usepackage[inline]{enumitem}
\usepackage[colorlinks]{hyperref}
\hypersetup{linkcolor=black,citecolor=blue,filecolor=black,urlcolor=black} 
\usepackage{array}
\usepackage{adjustbox}
\usepackage{tabularx}
\usepackage{multirow}
\usepackage{multicol}
\usepackage{arydshln}
\usepackage{comment}
\usepackage{ifthen}
\usepackage{caption}
\usepackage{bbm}
\usepackage{enumitem}
\newcommand{\p}{\mathfrak{p}}

\renewcommand{\bar}[1]{\overline{#1}}

\DeclareMathOperator{\coker}{Coker}

\renewcommand{\epsilon}{\varepsilon}

\renewcommand{\hat}[1]{\widehat{#1}}

\DeclareMathOperator{\hgt}{ht}  			

\newcommand{\kk}{\mathbbm{k}}					        
\renewcommand{\L}{\mathcal{L}}
\DeclareMathOperator{\lcm}{lcm}
	
\DeclareMathOperator{\pd}{pd}  			
\renewcommand{\phi}{\varphi}

\DeclareMathOperator{\rank}{rank}
\DeclareMathOperator{\reg}{reg}			
\DeclareMathOperator{\rk}{rk}

\renewcommand{\setminus}{\smallsetminus}

\newcommand{\term}[1]{\textbf{\textsf{#1}}}

\renewcommand{\tilde}[1]{\widetilde{#1}}
\DeclareMathOperator{\Tor}{Tor}

\newtheorem{thm}{Theorem}[section]
\newtheorem{lemma}[thm]{Lemma}
\newtheorem{prop}[thm]{Proposition}
\newtheorem{cor}[thm]{Corollary}

\theoremstyle{definition}
\newtheorem{defn}[thm]{Definition}
\newtheorem{example}[thm]{Example}
\newtheorem{rmk}[thm]{Remark}
\newtheorem{question}[thm]{Question}

\numberwithin{equation}{section}
\numberwithin{figure}{section}

\title[Properties of LCM Lattices]{Properties of LCM Lattices of Monomial Ideals} 
\date{\today}
\author[M. Dorang]{Matthew Dorang}
\address{Purdue University, Department of Mathematics, West Lafayette, IN, USA}
\email{mdorang@purdue.edu}
\author[J. McCullough]{Jason McCullough}
\address{Iowa State University, Department of Mathematics, Ames, IA, USA}
\email{jmccullo@iastate.edu}

\begin{document}



\begin{abstract} LCM lattices were introduced by Gasharov, Peeva, and Welker as a way to study minimal free resolutions of monomial ideals.  All LCM lattices are atomic and all atomic lattices arise as the LCM lattice of some monomial ideal.  We systematically study other lattice properties of  LCM lattices.  For lattices associated to the edge ideal of a graph, we completely characterize the many standard lattice properties  in terms of the associated graphs: Boolean, modular,  upper semimodular, lower semimodular, supersolvable, coatomic, and complemented; edge ideals with graded LCM lattices were previously characterized by Nevo and Peeva as those associated to gap-free graphs.  For arbitrary monomial ideals, we prove the Cohen-Macaulayness of minimal monomial ideals associated to modular lattices.  We also prove separate necessary and sufficient lattice conditions for when the projective dimension of a monomial ideal matches the height of its LCM lattice.   Finally, we show that LCM lattices of Gorenstein edge ideals are coatomic and raise questions about the lattice properties of arbitrary Gorenstein monomial ideals.
\end{abstract}

\maketitle

\begin{spacing}{1.1}

\section{Introduction}

Let $\kk$ be a field and let $S = \kk[x_1,\ldots,x_n]$ denote a polynomial ring over $\kk$.  
Free resolutions of graded ideals and modules over $S$ have been studied intensely since the time of Hilbert.  Even for monomial ideals, the structure of minimal free resolutions is not entirely understood.  Resolutions of monomial ideals are of particular interest both because of their combinatorial nature and because deformation arguments lead to statements about arbitrary graded ideals; see e.g. \cite[Theorem 22.9]{Peeva11}.  
Rather than computing all of the differentials in a minimal free resolution, one can instead compute just the graded Betti numbers 
$\beta_{i,j}^S(M) = \dim_\kk \Tor_i^S(M,\kk)_j$, 
which count the ranks of the graded free modules in a minimal free resolution of $M$.  
Gasharov, Peeva, and Welker introduced the LCM lattice of a monomial ideal \cite{GPW99} and proved how to compute graded Betti numbers in terms of the simplicial homology of order complexes of intervals in the LCM lattice.  Thus, understanding the structure of the LCM lattice of a monomial ideal provides insights into the structure of its minimal free resolution.

In this paper, we initiate a thorough study of lattice properties of LCM lattices of monomial ideals.  While classifying arbitrary monomial ideals whose LCM lattices have specific properties may be a very difficult problem, we accomplish this for squarefree monomial ideals generated in degree 2, which can be thought of as edge ideals of a graph.  Namely, given a finite, simple graph $G = (V,E)$ on vertex set $V = \{1,\ldots,n\}$, its edge ideal $I(G)$ is the ideal
\[I(G) = \left( x_ix_j \mid \{i,j\} \in E\right) \subseteq \kk[x_1,\ldots,x_n].\]
Since the LCM lattice of a monomial ideal and its polarization are isomorphic \cite[Proposition 6.1]{Phan06}, we get information about the LCM lattices of all monomial ideals generated in degree 2.  It is known that all LCM lattices are atomic and all atomic lattices arise as LCM lattices of some monomial ideal, though not uniquely.  There is a process by which one can associate a unique minimal squarefree monomial ideal to an atomic lattice by way of Phan's algorithm; see Theorem~\ref{Phan}.  Nevo and Peeva previously gave a sufficient condition for edge ideals with graded LCM lattices as those associated to gap-free graphs.  

One main purpose of this paper is to classify standard lattice properties of LCM lattices of edge ideals in terms of the associated graphs.  In Section~\ref{section:edge:ideals}, we characterize the following properties for such LCM lattices: Boolean, modular, upper semimodular, lower semimodular, supersolvable, coatomic, and complemented; see Theorems~\ref{GradedGraph}, \ref{USS},  \ref{thm:supersolvable},  \ref{thm:lsm}, \ref{thm:coatomic:lcm:lattices}, Corollary~\ref{BooleanEdgeIdeal}, and Proposition~\ref{prop:complemented:LCM:lattice}.  We summarize these results in a Venn diagram of possible combinations.  

In the case of arbitrary monomial ideals, it is difficult to decide if a given monomial ideal is minimal without running Phan's algorithm.  Nonetheless, we do classify monomial ideals with Boolean LCM lattices in Theorem~\ref{Boolean:LCM:lattice}.  We also prove that minimal monomial ideals of modular lattices are Cohen-Macaulay in Theorem~\ref{Modular}.  We consider when the projective dimension of a monomial ideal matches the height of its LCM lattice in Section~\ref{sec:pdim}.  Finally, in Section~\ref{sec:Gorenstein}, we consider the properties of LCM lattices of Gorenstein monomial ideals.

\section{Background}\label{sec:background}
In this section, we review the basic results and terminology needed in the rest of this paper.

\subsection{Lattices}

 A \term{lattice} is a partially ordered set (poset) $(L,\leq)$, where every pair of elements admits a greatest lower bound, or meet, and a least upper bound, or join.  We denote the meet and join as $x \wedge y$ and $x \vee y$, respectively. In this paper we deal exclusively with finite lattices.  A lattice is \term{bounded} if it admits both a maximum element, denoted by $\hat{1}$, and minimum element,  denoted by $\hat{0}$.  All finite lattices are bounded.

 An element $y \in L$ is said to \term{cover} an element $x \in L$ if $x < y$ and, whenever $x \leq z \leq y$ for some element $z \in L$, either $x = z$ or $y = z$.  A lattice may be visualized by way of its \term{Hasse diagram}, in which larger elements are placed above smaller ones with line segements between an element and any it covers.
    A totally ordered subset $C$ of elements in a lattice is called a \term{chain}. 
    A chain $C$ is called \term{saturated} if between any elements $x,y \in C$ such that $x < y$, either $y$ covers $x$, or there exists some $z \in C$ with $x < z < y$. 
    We say that a maximal chain is of length $n$ if there are $n+1$ elements in $M$.
    The \term{height} of $L$, which we denote by $\operatorname{ht}(L)$, is the maximum of lengths of chains of $L$.

 An element $x \in L$ is called an \term{atom} if it covers $\hat{0}$. It is called a \term{coatom} if it is covered by $\hat{1}$. 
    Let $\operatorname{at}(L)$ denote the set of atoms of $L$, and $\operatorname{coat}(L)$ the set of coatoms. Following Stanley \cite{Stanley11}, we call $L$ \term{atomic} if every element can be written as a finite join of atoms and \term{coatomic} if every element is a finite meet of coatoms; other authors call such lattices atomistic and coatomistic, respectively.  An element $x \in L$ is called \term{meet-irreducible} if $x \neq \hat{1}$ and whenever $x = a \wedge b$ for $a,b \in L$, we have that $x = a$ or $x = b$. Likewise, an element $x \in L$ is called \term{join-irreducible} if $x \neq \hat{0}$ and whenever $x = a\vee b$ for $a,b \in L$, we have that $x = a$ or $x = b$.
    We let $\operatorname{mi}(L)$ denote the set of meet-irreducible elements.  Observe also that in a finite atomic lattice, the join-irreducibles are precisely the atoms.

    A bounded lattice $L$ is called \term{complemented} if for every element $x \in L$, there exists a \term{complement} element, that is, an element $y \in L$ such that $x \wedge y = \hat{0}$ and $x \vee y = \hat{1}$. $L$ is called \term{uniquely complemented} if there is exactly one complement element for each element. 

    A lattice $L$ is called \term{graded} if there exists an order-preserving function $\rho : L \to \mathbb{Z}$, where for all $x,y \in L$ with $y$ covering $x$, we have that $\rho(y) = \rho(x) + 1$.  For a finite lattice, we write $\rk(x) = \rho(x)$, called the \term{rank} of $x$, and can insist that $\rk(\hat{0}) = 0$.

    Now suppose that $L$ is a graded lattice.  Then $L$ is \term{upper semimodular} if $\rk(x) + \rk(y) \geq \rk(x \wedge y) + \rk(x \vee y)$ for all $x,y, \in L$; $L$ is \term{lower semimodular} if $\rk(x) + \rk(y) \leq \rk(x \wedge y) + \rk(x \vee y)$ for all $x,y, \in L$; and $L$ is \term{modular} if $\rk(x) + \rk(y) = \rk(x \wedge y) + \rk(x \vee y)$ for all $x,y, \in L$.  Clearly $L$ is modular if and only if it is both upper and lower semimodular.  $L$ is said to be \term{supersolvable} if there exists a maximal chain $M$, where for all $m \in M$ and $x \in X$,
 $
        \rk(x) + \rk(m) = \rk(x \wedge m) + \rk(x \vee m).
   $
    Such a maximal chain is called a \term{modular chain}.  The lattice $L$ is \term{geometric} if it is atomic and upper semimodular. A lattice is geometric  if and only if it is the lattice of flats of a matroid, a combinatorial object generalizing the notion of linear independence in vector spaces; see e.g.  \cite[p. 106]{Birkhoff48}.  We shall not need the language of matroids much and refer the interested reader to \cite{Welsh76}.

    A lattice $L$ is said to be \term{distributive} if for all $x,y,z \in L$, we have that
$
        x \wedge (y\vee z) = (x\wedge y)\vee (x\wedge z).
$
    A bounded lattice $L$ is called \term{Boolean} if it is both distributive and uniquely complemented.  A finite, atomic, distributive lattice is Boolean \cite[p. 292]{Stanley11} as is a finite, atomic, uniquely complemented lattice \cite[p. 170]{Birkhoff48}.  Other implications among finite atomic lattices are summarized in Figure~\ref{fig:lattice:implications}.  (USM and LSM denote upper semimodular and lower semimodular, respectively.)

       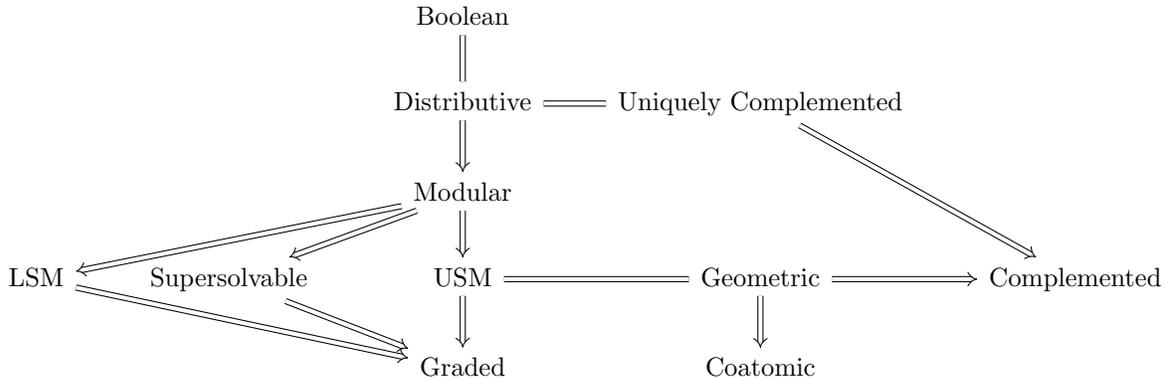
\begin{figure}[h]
    \begin{tikzcd}
&\text{Boolean} \arrow[d, Rightarrow, no head]                                 &      &                                  &                                                     \\
&\text{Distributive} \arrow[d, Rightarrow] \arrow[rr, Rightarrow, no head]     &   &                                      \text{Uniquely Complemented} \arrow[d, Rightarrow] \\
&\text{Modular} \arrow[d, Rightarrow] \arrow[dl, Rightarrow]   \arrow[rd, Rightarrow]                             &     &              \text{Complemented}                     &                                                     \\
\text{LSM} \arrow[rd, Rightarrow]&\text{Supersolvable} \arrow[d, Rightarrow] & \text{USM} \arrow[ld, Rightarrow] \arrow[r, Rightarrow, no head] & \text{Geometric}  \arrow[d, Rightarrow] \arrow[u, Rightarrow] &                                  \\
&\text{Graded}                                                       &  &  \text{Coatomic}                                    &                                                    
\end{tikzcd}
\caption{Diagram of Implications for Finite Atomic Lattices}\label{fig:lattice:implications}
\end{figure}

\noindent The one implication in Figure~\ref{fig:lattice:implications} that is not immediate from the definitions is that geometric lattices are complemented; the proof for this can be found in \cite[p. 288]{Stanley11}, where geometric implies a stronger condition called relatively complemented, which we do not define here. 

Given two lattices $L_1, L_2$, their product $L_1 \times L_2$ is naturally a lattice with order determined by $(x_1,y_1) \le (x_2, y_2)$ if and only if $x_1 \le x_2$ in $L_1$ and $y_1 \le y_2$ in $L_2$.

\subsection{Monomial Ideals and Free Resolutions}

Let $\kk$ be a field and let $S = \kk[x_1,\ldots,x_n]$ be a polynomial ring over $\kk$.  We view $S$ as a standard graded ring with $\deg(x_i) = 1$ for all $i$.  Setting $S_i$ to be the $\kk$-vector space of all homogeneous polynomials, i.e. forms, of degree $i$, we can write $S = \bigoplus_{i \ge 0} S_i$.  Thus, $S$ is a $\mathbb{Z}$-graded ring with $S_i S_j \subseteq S_{i+j}$ for all $i, j \ge 0$.  We will also consider $S$ as a $\mathbb{Z}^n$-graded, or multigraded, ring and identify graded pieces of $S$ with the single monomial generator of that multidegree.  A module $M$ of $S$ is graded if $M = \bigoplus_{i \ge 0} M_i$ and $S_iM_j \subseteq M_{i+j}$.

A finitely generated graded module has a \term{minimal graded free resolution} $F_\bullet$ of $S$, meaning $F_i$ is a finitely generated graded free $S$-module for all $i$, the differentials $d_i:F_i \to F_{i-1}$ are graded maps of degree 0, $\mathrm{Im}(d_i) \subseteq (x_1,\ldots,x_n)F_{i-1}$, and $\coker(d_1) \cong M$.  Such resolutions are unique up to isomorphism of complexes and of length at most $n$ by Hilbert's Syzygy Theorem \cite[Theorem 15.2]{Peeva11}.  Since the modules $F_i$ are finitely generated and graded, we can write $F_i = \bigoplus_{j \in \mathbb{Z}} S(-j)^{\beta_{ij}(M)}$.  The numbers $\beta_{ij}(M)$, which can also be defined as $\beta_{ij}(M) = \dim_\kk \Tor_i^S(M,\kk)_j$, are the \term{graded Betti numbers} of $M$ and are invariants usually displayed in the Betti table in which $\beta_{ij}(M)$ is placed in column $i$ and row $j-i$.  The \term{projective dimension} of $M$ is defined as $\pd_S(M) = \max\{i\mid \beta_{ij}(M) \neq 0 \text{ for some } j\}$, which measures the length of the resolution.

The structure of free resolutions is of great interest, even for monomial ideals.  Note that monomial ideals are precisely the $\mathbb{Z}^n$-graded ideals of $S$.  We can then define the \term{multigraded Betti numbers} of $M$ as $\beta_{i,m}(M) = \dim_\kk \Tor_i^S(M,\kk)_m$, where $m$ is a monomial of $S$ which we identify with its exponent vector in $\mathbb{Z}^n$.  The coarser graded Betti numbers can be recovered from the multigraded ones via 
\[\beta_{ij}(M) = \sum_{\substack{m \in \mathbb{Z}^n\\\deg(m) = j}} \beta_{i,m}(M).\]

Gasharov, Peeva, and Welker \cite{GPW99} defined the LCM lattice of a monomial ideal to study its free resolution.
By \cite[Proposition 3.3.1]{Stanley11}, to describe the LCM lattice, we need only define the poset structure and the joins.  Given a monomial ideal $I$ with minimal monomial generators $m_1,\ldots,m_q$, its \term{LCM lattice} $\L_I$ is the set of least common multiples of subsets of the generators, ordered by divisibility.  Joins are defined by taking lcm, and the lcm of the empty set is 1, which serves as the bottom element of the lattice.  Not only is $\L_I$ a lattice, it follows immediately from the definition that it is atomic.  Phan \cite{Phan06} showed that every atomic lattice may be realized as the LCM lattice of some monomial ideal; also see \cite[Corollary 8.4]{Berglund07}.  While this association is not bijective, there is a unique (up to isomorphism) squarefree monomial ideal whose LCM lattice matches a given finite atomic lattice.  The algorithm for constructing it goes as follows: label the meet-irreducible elements by variables; then label the atoms by the product of variables correspond to meet-irreducibles not above the chosen atom.  The monomial ideal generated by the monomials associated to the atoms then has LCM lattice isomorphic to the starting lattice.  This is summarized in the following theorem.

\begin{thm}[Phan {\cite[Theorem 5.1]{Phan06}, see also \cite[Theorem 58.6]{Peeva11}}]\label{Phan}
    Fix a field $\kk$. 
    Suppose $L$ is a finite, atomic lattice. Define
    \begin{align*}
        M(L) := \left( \prod_{\substack{a \in \operatorname{mi}(L)\\ b \not\leq a}}x_a \mid b \in \operatorname{at}(L)\right) \subset \kk[x_a \mid a \in \operatorname{mi}(L)] =: S(L).
    \end{align*}
    Then $\L_{M(L)} \cong L$.
\end{thm}

   \noindent We call the resulting monomial ideal $M(L)$ the \term{Phan ideal} of $L$, and we call a squarefree monomial ideal $I$ \term{minimal} if there exists some finite atomic lattice $L$ for which $I = M(L)$.  For other ways to associate monomial ideals to a given lattice, see \cite{Mapes13}.

The primary interest in LCM lattices is that they can be used to compute the multigraded Betti numbers of a monomial ideal.  To explain the result, we first recall the definition of a simplicial complex.  An \term{abstract simplicial complex} $\Delta$ on a finite set $X$ is a nonempty collection of subsets of $X$ closed under taking subsets.  An element $F \in \Delta$ is called a \term{face} of $\Delta$, and its dimension is $\dim(F) = |F| -1$.  Note that there is exactly one (-1)-dimensional face of $\Delta$ corresponding to the empty set. The case of interest for our purposes is that of the \term{order complex} $O(P)$ of a poset $P$, which is the simplicial complex on the same underlying set consisting of all chains of $P$.  

Given an abstract simplicial complex $\Delta$ and commutative ring $R$, we define the \term{simplicial chain complex} $C_\bullet(\Delta;R)$ by setting $C_i(\Delta;R)$ to the the free $R$-module with basis the set of $i$-dimensional faces of $\Delta$, and with differential determined by the formula $\partial_i(\{j_0,\ldots,j_i\}) = \sum_{k = 0}^i (-1)^i \{j_0,\ldots,\hat{j_k},\ldots,j_i\}$.  The $i$th \term{reduced simplicial homology} of $\Delta$ with coefficients in $R$ is then $\tilde{H}_i(C_\bullet(\Delta;R))$.  

The significance of the LCM lattice of a monomial ideal $I \subseteq S$ is captured by the following result, which shows that the multigraded Betti numbers of $S/I$ can be computed as the reduced simplicial homology of order complexes of intervals in the LCM lattice of $I$.

\begin{thm}[Gasharov-Peeva-Welker \cite{GPW99},  see {\cite[Theorem 58.8]{Peeva11}}]\label{LatticeHomology}
    Suppose $I \subseteq S$ is a monomial ideal. Then,
    \begin{align*}
        \beta_{i,m}^S(S/M) = \begin{cases}
            \operatorname{dim}_\kk \tilde{H}_{i-2}(O((1,m)_{\L_I});\kk) & \text{if $1 \neq m \in \L_I$}\\
            0 & \text{if $1 \neq m \notin \L_I$},
        \end{cases}
    \end{align*}
    where $(1,m)_{\L_I} := \{y \in \L_I : 1 <y < m\}$. 
\end{thm}

\subsection{Graphs and Edge Ideals}

A finite simple \term{graph} $G$ is a pair $(V,E)$ consisting of a finite set $V$ of \term{vertices} and set $E$ of unordered pairs of distinct vertices.  Thus, there are no loops and no multiple edges.  We say that $G$ is \term{nontrivial} if $E \neq \varnothing$. Given a graph $G$, its \term{complement} $G^c$ is the graph with the same vertex set and with edge set determined by $\{i,j\} \in E(G^c)$ if and only if $\{i,j\} \notin E(G)$.  Given a graph $G = (V,E)$, a \term{subgraph} is a graph $H = (G',E')$ where $G' \subseteq G$ and $E' \subseteq E$.  The subgraph $H$ is called \term{induced} if every edge $e \in E(G)$ with $e \subseteq V(H)$ is an edge of $H$. For a graph $H$, we say that $G$ is $H$-free if there exists no induced subgraph of $G$ isomorphic to $H$.  The \term{complete graph} $K_n$ is defined as the graph on vertices $[n] = \{1,\ldots,n\}$ with all possible edges.  The \term{path graph} $P_n$ is the graph on $[n]$ with edge set $E(G) = \{\{1,2\},\{2,3\},\ldots,\{n-1,n\}\}$.  The \term{cycle graph} $C_n$ is given by $V = [n]$ and $E = \{\{1,2\},\{2,3\},\dotsc ,\{n-1,n\},\{n,1\}\}$. The \term{star graph} $\operatorname{St}_n$ is given by $V = [n]$ and $E = \{\{1,i\} : i \in [n] - \{1\}\}$.
Given two vertices $v,w \in V(G)$, we say that $v$ and $w$ are \term{adjacent} if $\{v,w\} \in E(G)$.  The \term{neighbors} of $v$ is the set $N(v) = \{y \in V(G) \mid \{v,y\} \in E(G)\}$. 

Let $G$ be a graph on vertex set $V = [n]$ and fix a field $\kk$.  The \term{edge ideal} of $G$ is the ideal
\[I(G) = \left( x_i x_j \mid \{i,j\} \in E(G)\right) \subseteq S = \kk[x_1,\ldots,x_n].\]
Resolutions of edge ideals have received a lot of attention, and we refer the reader to \cite{Herzog11} for an introduction.

\section{Monomial ideals with Boolean LCM Lattices}\label{sec:boolean}

In this section, we characterize those monomial ideals whose LCM lattices are Boolean.  We first recall The Taylor resolution of a monomial ideal.

\begin{defn}
Fix a polynomial ring $S$.
    Let $I \subset S$ be a monomial ideal, and $\{m_1,\dotsc,m_q\}$ its minimal monomial generating set. Let $\Lambda$ denote the exterior algebra over $\kk^q$ with standard basis elements denoted $e_1,\dotsc,e_q$. We define $\mathbf{T}_\bullet$ to be the $S$-module $S\otimes_\kk \Lambda$ with homological grading given by $\operatorname{hdeg}(e_{j_1}\wedge \dotsb \wedge e_{j_i}) = i$, and with differential $d : \mathbf{T}_\bullet \to \mathbf{T}_\bullet $ given by
    \begin{align*}
        &d(e_{j_1}\wedge \dotsb \wedge e_{j_i})\\
        &= \sum_{1 \leq p \leq i}(-1)^{p-1}\frac{\operatorname{lcm}(m_{j_1},\dotsc,m_{j_i})}{\operatorname{lcm}(m_{j_1},\dotsc,\widehat{m}_{j_p},\dotsc,m_{j_i})} \otimes e_{j_1}\wedge\dotsb\wedge\widehat{e}_{j_p}\wedge\dotsb\wedge e_{j_i},
    \end{align*}
    where $\widehat{e}_{j_p}$ and $\widehat{m}_{j_p}$ denotes omission of the respective element.
    For each $i \in \mathbb{N}$, let $\Lambda_i$ denote the subspace of $\Lambda$ spanned by elements of the form $e_{j_1}\wedge\dotsb \wedge e_{j_i}$ and write $\mathbf{T}_i := S\otimes_\kk \Lambda_i$.
    The \term{Taylor resolution} is given by
    \begin{align*}
        \mathbf{T}_\bullet: 0 \longrightarrow \mathbf{T}_{q}\overset{d_q}{\longrightarrow} \mathbf{T}_{q-1} \overset{d_{q-1}}{\longrightarrow} \dotsc \longrightarrow \mathbf{T}_1 \overset{d_{1}}{\longrightarrow} \mathbf{T}_0 .
    \end{align*}
\end{defn}

\begin{thm}[{\cite[Theorem 26.7]{Peeva11}}]
    The Taylor complex $\mathbf{T}_\bullet$ is a free resolution of $S/I$.
\end{thm}

While the Taylor resolution is a free resolution, it is often not minimal. The following theorem characterizes when the Taylor resolution of $S/I$ is minimal. The proof of \textit{(2)} $\Leftrightarrow$ \textit{(3)} appears in \cite{Kuei-Nuan_Jason_13}.

\begin{thm}\label{Boolean:LCM:lattice}
    Let $I \subseteq S$ be a monomial ideal with minimal monomial generators $m_1,\ldots,m_q$. The following are equivalent:
\begin{enumerate}
        \item $\L_I$ is a Boolean lattice
        \item For each $i$ there is a variable $x_{j_i}$ and a positive integer $n_i$ such that $x_{j_i}^{n_i} \mid m_i$ and $x_{j_i}^{n_i} \!\not| \,\, m_j$ for $j \neq i$. 
        \item The Taylor resolution $\mathbf{T}_\bullet$ is minimal.
        \item $\pd_S(S/I) = \mu(I)$,
    \end{enumerate}
\end{thm}

\begin{proof}
Since polarization does not affect the LCM lattice, we may assume that $I$ is squarefree.  In this case, the condition implies that each monomial generator is divisible by a unique variable.  Then \textit{(2)} $\Leftrightarrow$ \textit{(3)} follows from \cite[Proposition 4.1]{Kuei-Nuan_Jason_13}.    \\

\noindent \underline{\textit{(2)}  $\Rightarrow$ \textit{(1)}}: Again suppose that $I$ is squarefree and that each generator $m_i$ of $I$ is divisible by a unique variable, say $x_i$.  Then we may identify the element $\lcm(m_{i_1},\ldots,m_{i_t}) \in \L_I$ with the element $\{x_{i_1},\ldots,x_{i_t}\}$ in the Boolean lattice $2^{X}$, where $X = \{x_1,\ldots,x_n\}$.  It is easy to see that this is a lattice isomorphism.\\

\noindent \underline{\textit{(1)} $\Rightarrow$ \textit{(4)}}: By Theorem~\ref{LatticeHomology}, $\pd_S(S/I) = \max\{i \mid \tilde{H}_{i-2}(O((1,m)_{\L_I};\kk) \neq 0\}$.  By  \cite[Theorem 4.1]{Folkman66}, since $\L_I$ is in particular geometric, $\tilde{H}_{q-2}(O((1,m)_{\L_I};\kk) \neq 0$, where $q = \mu(I)$ is the length of the Taylor resolution.  Thus, $\pd_S(S/I) = q$.\\

\noindent \underline{\textit{(4)} $\Rightarrow$ \textit{(3)}}: We prove the contrapositive.  Write $I = (m_1,\ldots,m_q)$.  Suppose that $\mathbf{T}_\bullet$ is not minimal.  Then some entry in one of the differentials is a unit, say $\frac{\operatorname{lcm}(m_{j_1},\dotsc,m_{j_i})}{\operatorname{lcm}(m_{j_1},\dotsc,\widehat{m}_{j_p},\dotsc,m_{j_i})} = 1$.  But then $\frac{\operatorname{lcm}(m_{1},\dotsc,m_{q})}{\operatorname{lcm}(m_{1},\dotsc,\widehat{m}_{j_p},\dotsc,m_{q})} = 1$ as well, meaning the differential $d_q$ is not minimal.  It follows that the minimal free resolution of $S/I$, which is a subcomplex of $\mathbf{T}_\bullet$, has length $\pd_S(S/I) < \mu(I) = q$.
\end{proof}

\section{Cohen-Macaulayness of Minimal Monomial Ideals Associated to Modular Lattices}\label{sec:modular}

The purpose of this section is to prove that the Phan ideal of a modular lattice is always Cohen-Macaulay.  Recall that a graded module over a standard graded polynomial ring $S$ is called \term{Cohen-Macaulay} if $\dim(M) = \mathrm{depth}(M)$; here $\dim(M)$ denotes the Krull dimension of $M$, and $\mathrm{depth}(M)$ is the depth of $M$ in the graded maximal ideal $(x_1,\ldots,x_n)$ of $S$.  By the Auslander-Buchsbaum Formula \cite[Theorem 15.3]{Peeva11}, an ideal defines a Cohen-Macaulay cyclic module $S/I$ if and only if $\pd_S(S/I) = \hgt(I)$.

\begin{thm}\label{Modular}
    Suppose $I$ is a minimal monomial ideal with a modular LCM lattice. Then $S/I$ is Cohen-Macaulay. 
\end{thm}

\noindent We first recall the Birkhoff's classification of finite atomic, modular lattices.

\begin{thm}[{\cite[p. 120]{Birkhoff48}}]\label{Birkhoff}
    Suppose $L$ is a finite, atomic, modular lattice. Then, $L \cong L_1\times \dotsb \times L_t$, where each $L_i$ is isomorphic to one of the following lattices:
\begin{enumerate}
        \item The lattice of subspaces of a finite vector space.
        \item The incidence lattice of a finite projective plane.
        \item A finite lattice of height $1$ or $2$.
    \end{enumerate}
\end{thm}

\noindent Note that the projective planes in the previous theorem include non-Desarguesian planes, such as the Hughes plane.  Here is the formal definition.

\begin{defn}
    A \term{projective plane} is an ordered triple $\mathcal{H} = (P,L,I)$.  The elements of $P$ are called \term{points}; the elements of $L$ are called \term{lines}, and $I$ is a symmetric relation on $P \cup L$ recording the the incidence relation satisfying the following properties:
\begin{enumerate}
        \item Given any two points $p_1,p_2 \in P$, there exists a unique line $\ell \in L$ with $p_1,p_2 \in \ell$.
        \item Given any two lines $\ell_1,\ell_2 \in L$, there exists a unique point $p \in P$ for which $p \in \ell_1\cap \ell_2$.
        \item There exists four points $p_1,p_2,p_3,p_4 \in P$ of which no line contains more than $2$.
    \end{enumerate}
\end{defn}  

\noindent While the classification of projective planes is an open question, there are restrictions on their possible sizes; see e.g. \cite[p. 111]{Birkhoff48}.


To prove that minimal monomial ideals associated to modular lattices are Cohen-Macaulay, we must show that those associated to projective planes and finite vector spaces are Cohen-Macaulay.  We first observe that the height of an LCM lattice bounds the projective dimension of the associated monomial ideal.

\begin{lemma}\label{pdHeight}
    Suppose $I$ is a monomial ideal. Then $\pd_S(S/I) \leq \operatorname{ht}(\L_I))$.
\end{lemma}
\begin{proof}
    Let $r := \operatorname{ht}(\L_I)$, and
    suppose $i > r$.
    Then by Theorem~\ref{LatticeHomology} we have the following:
    \begin{align*}
        \beta_{i}^{S}(S/M) 
        &=\sum_{1 \neq m \in \L_I}\beta_{i,m}^{S}(S/M)\\
        &=\sum_{1 \neq m \in \L_I}\dim \tilde{H}_{i-2}(O((1,m)_{L(M)});\kk).
    \end{align*}

    \noindent Recall that $\tilde{H}_{i-2}(O((1,m)_{\L_I});\kk)$ is a quotient of the vector space generated by $i-2$ simplices of $O((1,m)_{\L_I})$. However, the largest maximal chain in $L$ has size equal to $\operatorname{rk}(L) = r+1$. Hence, the cardinality of simplices in $O((1,m)_{\L_I})$ is bounded above by $r-1$. In particular, the dimension of simplices in $O((1,m)_{\L_I})$ is bounded by $r-2 < i-2$. As such, $ \tilde{H}_{i-2}(O((1,m)_{\L_I});k) = 0$ for all $m \in \L_I$. Hence, $\beta_{i}^{S}(S/M) = 0$ for all $i > \hgt(\L_I)$, and thus, $\pd(S/M) \le \hgt(\L_M)$.
\end{proof}

\noindent The previous result was also proved by Phan in \cite[Theorem 6.4]{Phan06}.  Phan also proved that $\pd_S(S/I)$ is at most the width (i.e. maximal size of an antichain) of the subposet of meet-irreducibles of $\L_I$.  In the next section we explore when equality occurs in the previous lemma.

\begin{lemma}\label{subspacesModular}
    Suppose $L$ is the lattice of subspaces of a finite vector space $V = \mathbb{F}_q^r$, and let $I \subseteq S = \kk[x_1,\ldots,x_n]$ be the associated minimal monomial ideal. Then $S/I$ is Cohen–Macaulay. 
\end{lemma}

\begin{proof} By Lemma~\ref{pdHeight}, it suffices to show that for all minimal primes $\p$ of $I$ we have $\hgt(\p) = \hgt(L) = r$.  Since $L$ is modular, it is geometric and hence coatomic by \cite[Proposition 1.7.8]{Oxley92}; therefore all meet-irreducible elements are coatoms, which correspond to hyperplanes in $V$.  It follows from Theorem~\ref{Phan} that the variables $x_1,\ldots,x_n$ correspond to the hyperplanes of $V$ and a minimal monomial generator of $I$ corresponds to an atom of $L$, which in turn corresponds to a point $p$ in $V$ and is labeled by products of variables corresponding to hyperplanes not containing $p$.  Thus, a minimal prime of $I$ is generated by a minimal set of variables corresponding to a set $\mathcal{H}$ of hyperplanes with the property that for any point $p$ there is a hyperplane $H \in \mathcal{H}$ with $p \notin H$.  Since the intersection of any $r-1$ hyperplanes contains at least one point, any minimal prime of $I$ has height at least $r$.  It follows that $\hgt(I) = \pd(S/I) = r$ and $S/I$ is Cohen-Macaulay.
\end{proof}

\noindent The proof in the case of projective planes is similar.

\begin{lemma}\label{projectivePlaneModular}
    Suppose $L$ is the incidence lattice of a finite projective plane. Then $S(L)/M(L)$ is Cohen–Macaulay.
\end{lemma}

\begin{proof} Once again, since $\hgt(L) = 3$, we must prove that all minimal primes are generated by $3$ variables.  Now the atoms are the points and the coatoms are the lines.  Since any two lines meet in a point, any minimal prime must be generated by at least (hence exactly) 3 varables.  Thus, $\hgt(I) = \pd(S/I) = 3$ and $S/I$ is Cohen-Macaulay.
\end{proof}

\begin{lemma}\label{height2lattices}
    Suppose $L$ is a graded (i.e. modular) lattice of height $2$, and let $I \subseteq S = \kk[x_1,\ldots,x_n]$ be the associated minimal monomial ideal.  Then $S/I$ is Cohen-Macaulay.
\end{lemma}

\begin{proof} In this case the atoms, coatoms, and meet-irreducible elements coincide.  Applying Phan's algorithm shows that each monomial generator of $I$ is a product of all but one of the variables.  Thus, any ideal generated by two distinct variables is a minimal prime of $I$ while no single variable divides every generator.  It follows that $\hgt(I) = \pd_S(S/I) = 2$ and that $S/I$ is Cohen-Macaulay.
\end{proof}

\begin{rmk}
    In fact, it is not hard to see that the ideal in Lemma~\ref{height2lattices} has a Hilbert-Burch resolution associated to the $(n+1) \times n$ matrix:
    \[
    \begin{pmatrix}
        x_1&0&0&\cdots&0&0\\
        -x_2&x_2&0&\cdots&0&0\\
        0&-x_3&x_3& &0&0\\
        \vdots & & \ddots & \ddots &&\vdots\\
        0&0&0 &\cdots & -x_{n-1} & x_{n-1}\\
        0&0&0&\cdots&0&-x_n\\
    \end{pmatrix}
    \]
\end{rmk}

We can now prove our main result of this section.

\begin{proof}[Proof of Theorem~\ref{Modular}]
    Since $I$ is minimal, we identify it with the Phan ideal of its LCM lattice $\L_I$.
    By assumption $\L_I$ is modular. Hence, by Theorem~\ref{Birkhoff}, $\L_I \cong L_1 \times \dotsb \times L_k$ for modular lattices $L_1,\dotsc,L_k$, where each is a lattice of a finite vector space, an incidence lattice of a finite projective plane, or a lattice of height $1$ or $2$.  Writing $I_i$ for the Phan ideal of $L_i$ and $S_i$ the associated polynomial ring, note that $S/I \cong S_1/I_1 \otimes_\kk \cdots \otimes_\kk S_t/I_t$.  By \cite[Theorem 2.1]{BK02}, it suffices to show that $S_i/I_i$ is Cohen-Macaulay for each $i$.  The cases where $L_i$ is a lattice associated to a finite vector space, a projective plane, or a height 2 graded lattice are covered by Lemmas~\ref{subspacesModular}, \ref{projectivePlaneModular}, \ref{height2lattices}, respectively.  The case where $L_i$ has height one is trivial as $S_i/I_i \cong \kk[x]/(x) \cong \kk$.  The result follows.
\end{proof}

After discussing this result with Vinh Nguyen, we came across another method to prove Theorem~\ref{Modular}, which perhaps explains why the modular lattice hypothesis is critical.  First, we need the notion of a matroid cover ideal.  Let $M$ be a simple matroid on $E = \{1,\ldots,n\}$, considered as a simplicial complex.  The \term{cover ideal} of $M$ is the squarefree monomial ideal
\[J(M) = \bigcap_{F \in \mathcal{F}(M)} \mathfrak{p}_{F} \subseteq \kk[x_1,\ldots,x_n],\]
where $\mathcal{F}(M)$ denotes the facets of the simplicial complex $M$ and $\mathfrak{p}_F = \left( x_i \mid i \in F\right) \subseteq \kk[x_1,\ldots,x_n]$.  This is exactly the Stanley Reisner ideal corresponding to the dual $M^\ast$ of the matroid $M$.  For definitions related to matroids, we refer the reader to \cite{Welsh76}.

Next we prove the rather astonishing fact is that the matroid cover ideal of a simple matroid and the Phan ideal of its lattice of flats coincide precisely when the lattice is modular.

\begin{thm}\label{cover} Let $M$ be a simple matroid whose lattice of flats $L = \mathcal{L}(M)$ is modular.  Then the Phan ideal of $L$ and the cover ideal of $M$ coincide; that is
\[M(L) = J(M).\]
In particular, $M(L)$ is Cohen-Macaulay.
\end{thm}

\begin{proof}
 By \cite[Proposition 2.9]{MN24}, the minimal generators of $J(M)$ are exactly the monomials $x_{[n] \smallsetminus H}$, where $H$ is a hyperplane of $M$.  

Now consider the Phan ideal of the lattice of flats $\mathcal{L}$ of $M$.  Since $L$ is modular, the meet-irreducible elements are precisely the coatoms.   The coatoms in $\mathcal{L}$ are precisely the hyperplanes in $M$.   Also since $L$ is modular, $L$ is self-dual and we can identify coatoms with atoms.  By Phan's algorithm, the generators of the minimal monomial ideal associated to $\mathcal{L}$ correspond precisely to the complements of the hyperplanes by the above identification.  Thus the ideals are the same.  The Cohen-Macaulay property then follows from \cite[Theorem 2.1]{Varbaro11} or \cite[Theorem 3.5]{MT11}.
\end{proof}

 When a geometric lattice $L$ is not modular, then it must have more coatoms than atoms by \cite[Theorem 2]{greene_1970}.  Thus the cover ideal of a matroid and the Phan ideal of its (geometric) lattice of flats are never equal outside the modular setting - they live in polynomial rings with different numbers of variables.

 Note that while the Cohen-Macaulay property of a monomial ideal can depend on the characteristic of the coefficient field \cite[Example 12.4]{Peeva11}, the previous results are characteristic-free.  However, it only applies to minimal monomial ideals.  Indeed, we could simply take a minimal monomial ideal of height at least two with modular LCM lattice and multiply all generators by a new variable; this would not change the associated LCM lattice but the new ideal would then have height one and thus would no longer be Cohen-Macaulay.

In the following example, we consider the simplest example of a projective plane, describe its lattice ideal, the LCM lattice, and some properties of the minimal free resolution.

\begin{example}
    The Fano Plane, depicted in Figure~\ref{Fano-Plane}, is the unique projective plane of order $2$, making it the smallest projective plane. Each line contains $3$ points and each point is incident to $3$ lines.
    \begin{figure}[ht]
        \centering
    \begin{tikzpicture}
        \draw (0,1*1.732) circle (1*1.732);
        \node (1) at (0,0) [circle, draw, fill] {};
        \node (2) at (0,3*1.732) [circle, draw, fill] {};
        \node (3) at (3,0) [circle, draw, fill] {};
        \node (4) at (-3,0) [circle, draw, fill] {};
        \node (5) at (-1.5,1.5*1.732) [circle, draw, fill] {};
        \node (6) at (1.5,1.5*1.732) [circle, draw, fill] {};
        \node (7) at (0,1*1.732) [circle, draw, fill] {};
        \draw (1) -- (2);
        \draw (2) -- (3) -- (4) -- (2) -- cycle;
        \draw (3) -- (5);
        \draw (4) -- (6);
        \node at (0,3*1.732+0.35) {};
        \node at (-3-0.35,-0.35) {};
        \node at (3+0.35,-0.35) {};
        \node at (-1.5-0.35,1.5*1.732+0.35) {};
        \node at (1.5+0.35,1.5*1.732+0.35) {};
        \node at (0,-0.35) {};
        \node at (0+0.35,1*1.732+0.50) {};
    \end{tikzpicture}
   \caption{Fano Plane}
   \label{Fano-Plane}
    \end{figure}
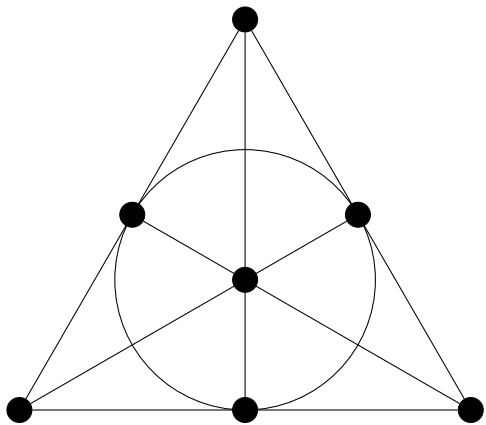
    The associated modular lattice is shown in Figure~\ref{FanoPlaneLattice}.  
     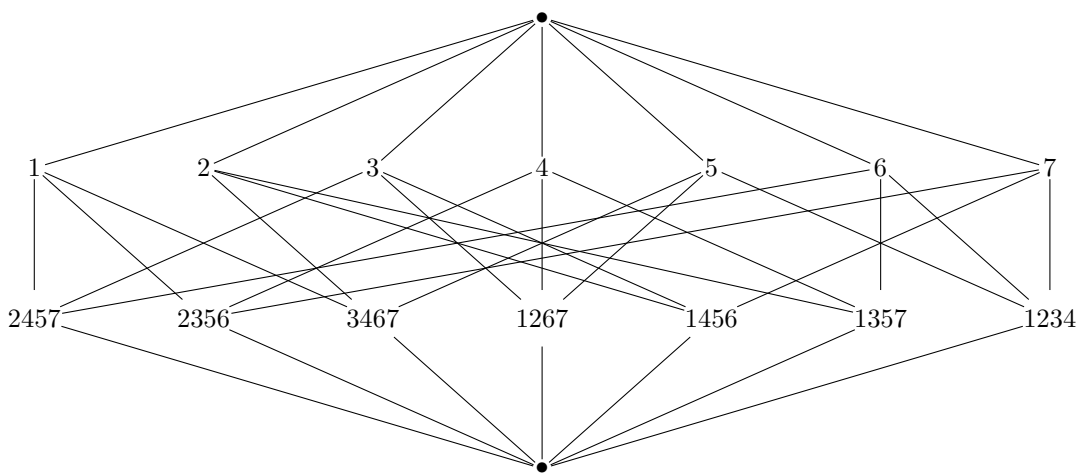
\begin{figure}[ht]
        \centering
         \begin{tikzpicture}[scale=1.5, vertices/.style={circle, inner sep=0pt}]
            \node [vertices] (U) at (0,4){$\bullet$};
             \node [vertices] (a) at (-4.5+0,1.33333){$2457$};
             \node [vertices] (b) at (-4.5+1.5,1.33333){$2356$};
             \node [vertices] (c) at (-4.5+3,1.33333){$3467$};
             \node [vertices] (d) at (-4.5+4.5,1.33333){$1267$};
             \node [vertices] (e) at (-4.5+6,1.33333){$1456$};
             \node [vertices] (f) at (-4.5+7.5,1.33333){$1357$};
             \node [vertices] (g) at (-4.5+9,1.33333){$1234$};
             \node [vertices] (7) at (-4.5+0,2.66667){$1$};
             \node [vertices] (6) at (-4.5+1.5,2.66667){$2$};
             \node [vertices] (5) at (-4.5+3,2.66667){$3$};
             \node [vertices] (4) at (-4.5+4.5,2.66667){$4$};
             \node [vertices] (3) at (-4.5+6,2.66667){$5$};
             \node [vertices] (2) at (-4.5+7.5,2.66667){$6$};
             \node [vertices] (1) at (-4.5+9,2.66667){$7$};
             \node [vertices] (T) at (0,0){$\bullet$};
             \foreach \to/\from in {2/a, 5/a, 7/a, 1/b, 4/b, 7/b, 3/c, 6/c, 7/c, 3/d, 4/d, 5/d, 1/e, 5/e, 6/e, 2/f, 4/f, 6/f, 1/g, 2/g, 3/g, a/T, U/1, b/T, U/2, c/T, U/3, d/T, U/4, e/T, U/5, f/T, U/6, g/T, U/7}
     \draw [-] (\to)--(\from);
     \end{tikzpicture}
        \caption{Lattice of Flats for the Fano Plane}
        \label{FanoPlaneLattice}
    \end{figure}
    The coatoms, which are all of the meet-irreducible elements, are labeled by distinct variable indices.  The atoms, corresponding to the generators of the associated minimal monomial ideal, are labeled by products of variables corresponding to coatoms not lying above it.
   The associated minimal monomial ideal $I$ is shown below: 
    \begin{align*}
        I = \left(x_2x_4x_5x_7,x_2x_3x_5x_6,x_3x_4x_6x_7,x_1x_2x_6x_7,x_1x_4x_5x_6,x_1x_3x_5x_7,x_1x_2x_3x_4\right).
    \end{align*}
   
    \noindent Using Macaulay2 \cite{M2}, the graded Betti table of $S/I$ is computed to be:
    \[\begin{array}{l|cccc}
         & 0 & 1 & 2 & 3\\
         \hline
      0 & 1 & - & - & -\\
      1 & - & - & - & -\\
      2 & - & - & - & -\\
      3 & - & 7 & - & -\\
      4 & - & - & 14 & 8
      \end{array}\]
      with each associated prime having $3$ generators. Thus, $S/I$ is Cohen-Macaulay, since the codimension and projective dimension are equal to $3$.  Moreover, this is a pure diagram, meaning in each column of  the Betti table contains a single nonzero entry.  Such diagrams are the basis of Boij-S\"oderberg decompositions of arbitrary Betti tables, and it is of interest to realize them as the Betti tables of modules  \cite{ES09}.  Modular lattices are thus homologically pure in the sense of \cite{FMS16}, which we now prove in general.
\end{example}

\begin{cor}
    Fix a prime power $q = p^n \in \mathbb{N}$. The Betti table of the minimal monomial ideal $I$ associated to the lattice of subspaces of a finite vector space $V = \mathbb{F}^r_q$ is pure with degree sequence 
    \[d = (0,q^{r-1},q^{r-1}+q^{r-2},\ldots,q^{r-1} + \cdots + q + 1).\]
\end{cor}

\begin{proof}
By Theorem~\ref{Modular}, $S/I$ is Cohen-Macaulay.  By Theorem~\ref{LatticeHomology}, the nonzero Betti numbers correspond to the degrees of monomials in the LCM lattice of $I$.  Note that there are $\frac{q^r-1}{q-1} = 1 + q + \cdots + q^{r-1}$ many lines and, dually, $\frac{q^r-1}{q-1}$ variables corresponding to the $\frac{q^r-1}{q-1}$ hyperplanes of $V$.  Elements of rank $i$, corresponding to subspaces of dimension $i$, are each contained in exactly $1+q+\cdots+q^{r-1-i}$ many hyperplanes.  Thus, the associated monomials have degree $(q^{r-1}-1) - (1 + q + \cdots + q^{r-1-i}) = q^{r-1} + q^{r-2} + \cdots + q^{r-i}$.  Since the lattice $L$ of subspaces of $V$ is atomic and modular, it is geometric.  Similarly, all intervals $[1,m]_L$ in $L$ are also geometric for all $m \in L$.  By \cite[Theorem 4.1]{Folkman66}, the reduced homology $\tilde{H}_{i-2}(O((1,m)_L;\kk)$ is only nonzero for $i = \rk(m)$.  The conclusion follows from Theorem~\ref{LatticeHomology}.
\end{proof}

The following example shows that the converse to Theorem~\ref{Modular} does not hold.

\begin{example}
    Take $S := \kk[x_1,\dotsc,x_6]$ and consider the edge ideal $I(G) \subseteq S$ of the following bipartite graph $G$. 
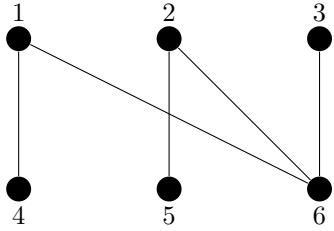
    \begin{align*}
    \begin{tikzpicture}
        \node (c) at (0,0) [circle, fill]{};
        \node (a) at (-2,0) [circle, fill]{};
        \node (e) at (2,0) [circle, fill]{};
        \node (d) at (0,-2) [circle, fill]{};
        \node (b) at (-2,-2) [circle, fill]{};
        \node (f) at (2,-2) [circle, fill]{};
        \node at (0,0.35) {2};
        \node at (-2,0.35) {1};
        \node at (2,0.35) {3};
        \node at (0,-2-.35) {5};
        \node at (-2,-2-.35) {4};
        \node at (2,-2-.35) {6};
        \draw (a) -- (b);
        \draw (a) -- (f);
        \draw (c) -- (d);
        \draw (c) -- (f);
        \draw (e) -- (f);
    \end{tikzpicture}
    \end{align*}
    The edge ideal is given by
    \begin{align*}
        I(G) = (x_1x_4, x_1x_6, x_2x_5, x_2x_6, x_3x_6).
    \end{align*}
    By Theorem~\ref{Phan}, $I(G)$ is the minimal monomial ideal for its LCM lattice.  One can observe that this graph comes from a poset on three elements $\{a,b,c\}$ with $a,b \leq c$ and $a,b$ incomparable. Thus, by \cite[Lemma 9.1.11]{Herzog11}, $S/I$ is Cohen-Macaulay.  However, since $G$ is not gap-free, $\L_{I(G)}$ is not graded by Theorem~\ref{GradedGraph}; in particular, it is not geometric.
\end{example}

The following example shows that the hypotheses of Theorem~\ref{Modular} cannot be significantly weakened.

\begin{example}
    
     Consider the following graph $G$:
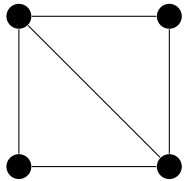
    \begin{align*}
        \begin{tikzpicture}
            \node (a) at (1,1) [circle, fill]{};
            \node (b) at (1,-1) [circle, fill]{};
            \node (c) at (-1,1) [circle, fill]{};
            \node (d) at (-1,-1) [circle, fill]{};
            \draw (a) -- (b);
            \draw (a) -- (c);
            \draw (b) -- (c);
            \draw (b) -- (d);
            \draw (c) -- (d);
        \end{tikzpicture}
    \end{align*}
    Let $M(G)$ denote the graphic matroid of $G$ and let $L = L(M(G))$ denote the associated lattice of flats, pictured in Figure~\ref{chordalGraphLattice}.
    \begin{figure}[ht]
        \centering
         \begin{tikzpicture}[scale=1.5, vertices/.style={circle, inner sep=0pt}]
            \node [vertices] (U) at (0,4){$12345$};
             \node [vertices] (b) at (-4.5+1.5,1.33333){$1$};
             \node [vertices] (c) at (-4.5+3,1.33333){$2$};
             \node [vertices] (d) at (-4.5+4.5,1.33333){$3$};
             \node [vertices] (e) at (-4.5+6,1.33333){$4$};
             \node [vertices] (f) at (-4.5+7.5,1.33333){$5$};
             \node [vertices] (6) at (-4.5+1.5-.75,2.66667){$123$};
             \node [vertices] (5) at (-4.5+3-.75,2.66667){$14$};
             \node [vertices] (4) at (-4.5+4.5-.75,2.66667){$24$};
             \node [vertices] (3) at (-4.5+6-.75,2.66667){$15$};
             \node [vertices] (2) at (-4.5+7.5-.75,2.66667){$25$};
             \node [vertices] (1) at (-4.5+9-.75,2.66667){$345$};
             \node [vertices] (T) at (0,0){$\varnothing$};
             \foreach \to/\from in {6/b, 6/c,  6/d, 6/c,  4/e, 4/c, 5/b, 5/e, 3/b, 3/f, 2/c, 2/f, 1/d, 1/e, 1/f, U/1, b/T, U/2, c/T, U/3, d/T, U/4, e/T, U/5, f/T, U/6}
     \draw [-] (\to)--(\from);
     \end{tikzpicture}
        \caption{Lattice of flats of graphic matroid associated to $G$ (labeled by edges of $G$)}
        \label{chordalGraphLattice}
    \end{figure}
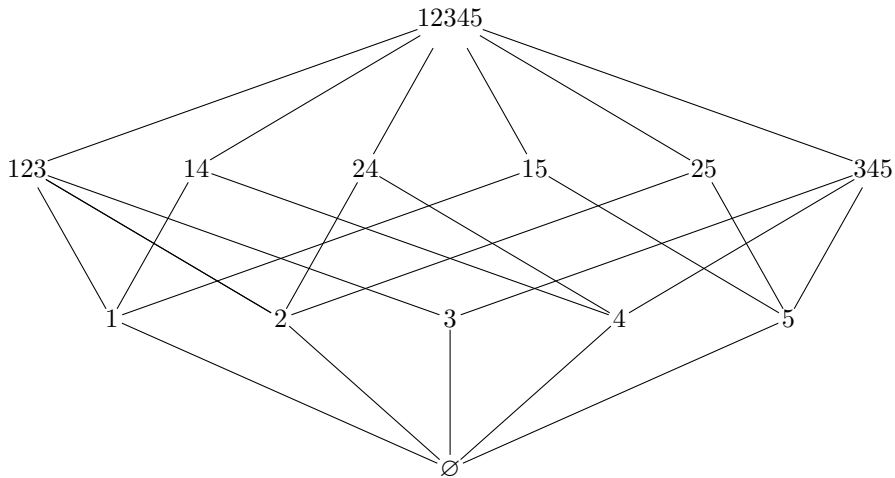
    By \cite[Theorem 1.7.5]{Oxley92}, $L$ is a geometric lattice.  By \cite{Stanley11}, it is supersolvable, since $G$ is a chordal graph.  It is also not hard to check that $L$ is complemented.  

    The associated minimal monomial ideal is:
    
    \begin{align*}      
        I = \left(x_{1}x_{3}x_{5},\,x_{3}x_{4}x_{6},\,x_{1}x_{2}x_{3}x_{4},\,x_{2}x_{4}x_{5},\,x_{1}x_{2}x_{6
      }\right) \subseteq k[x_1,\dotsc,x_6] =: S.
    \end{align*}
    It is easy to see that $(x_2,x_3)$ is a minimal prime and so $\hgt(I) = 2$.  Yet, by Theorem~\ref{pdim:Geometric:Lattices}, $\pd_S(S/I) = 3$.  Thus $S/I$ is not Cohen-Macaulay, even those $L$ is geometric, supersolvable, and complemented.
\end{example}

\section{Projective Dimension versus Height of the LCM Lattice}\label{sec:pdim}

As noted in Lemma~\ref{pdHeight}, one always has that the projective dimension of a monomial ideal is bounded above by the height of the associated LCM lattice.  In this section, we give necessary and sufficient conditions for when equality occurs.  

\begin{thm}\label{pdim:Geometric:Lattices}
Let $I \subseteq S$ be a monomial ideal. If $\L_I$ is geometric, then $\pd_S(S/I) =  \rank(\L_I)$.
\end{thm}

\begin{proof}
    If $L$ is a geometric lattice of rank $r \ge 2$, then by \cite[Theorem 4.1]{Folkman66}, $\tilde{H}_{r-2}(O((\hat{0},\hat{1})_{L});\mathbb{Z}) \cong \mathbb{Z}^{|\mu(\hat{0},\hat{1})|}$, while $\tilde{H}_{i}(O((\hat{0},\hat{1})_{L});\mathbb{Z}) = 0$ for $i \neq r-2$, where $\mu(x,y)$ denotes the M\"obius function of the interval $(x,y)$ in $L$.  By \cite[Theorem 4]{Rota64}, since $L$ is geometric, $\mu(x,y) \neq 0$ for any $x \le y$.  The Universal Coefficients Theorem \cite[Theorem 9.32]{Rotman79} implies that $\tilde{H}_{r-2}(O((\hat{0},\hat{1})_{L});\kk)  \cong \tilde{H}_{r-2}(O((\hat{0},\hat{1})_{L});\mathbb{Z}) \otimes_\mathbb{Z} \kk \cong \kk^{|\mu(\hat{0},\hat{1})|} \neq 0.$  By Theorem~\ref{LatticeHomology}, $\beta_{r,m}^S(S/I) \neq 0$, where $m$ denotes the top monomial of $\L_I$.  It follows that $\pd_S(S/I) = r$.
\end{proof}

In \cite{Peeva02}, Peeva gives two constructions of the minimal free resolution of a monomial ideal with geometric LCM lattice.  One could equivalently derive Theorem~\ref{pdim:Geometric:Lattices} from \cite[Theorem 2.1]{Peeva02} and the fact that no-broken-circuit sets are independent sets and have cardinality equal to the rank of the associated lattice.

We also get the following dual statement to Theorem~\ref{pdim:Geometric:Lattices} by the same proof.

\begin{cor}\label{cor:lsm+coatomic}
Let $I \subseteq S$ be a monomial ideal.  If $\L_I$ is lower semimodular and coatomic, then $\pd_S(S/I) = \rank(\L_I)$.
\end{cor}

\begin{proof}
    If $\L_I$ is lower semimodular and coatomic, then the dual lattice $\L_I^\ast$ is geometric.  The proof is then the same as Theorem~\ref{pdim:Geometric:Lattices}.
\end{proof}

\noindent While there aren't many edge ideals with geometric LCM lattices, the previous corollary, combined with Theorems~\ref{thm:lsm} and \ref{thm:coatomic:lcm:lattices}, shows that complete graphs satisfy $\pd_S(S/I(K_n)) = \rank(\L_{I(K_n)}) = n-1$ for all $n \ge 2$.

The converses of Theorem~\ref{pdim:Geometric:Lattices} and Corollary~\ref{cor:lsm+coatomic} does not hold, as the following example shows.

    \begin{example}
    Consider the edge ideal of the graph pictured in Figure~\ref{fig:pdim=ht}.
    \begin{figure}[ht]
    \begin{center}
        \begin{tikzpicture}
            \node (0) at (0, 0) [circle,draw, fill,inner sep=0pt,minimum size=3pt] { };
            \node (1) at (-1, 1) [circle,draw, fill,inner sep=0pt,minimum size=3pt] { };
            \node (2) at (-1, -1) [circle,draw, fill,inner sep=0pt,minimum size=3pt] { };
            \node (3) at (1, 1) [circle,draw, fill,inner sep=0pt,minimum size=3pt] { };
            \node (4) at (1, -1) [circle,draw, fill,inner sep=0pt,minimum size=3pt] { };

            \draw (0) to (1);
            \draw (1) to (2);
            \draw (2) to (3);
            \draw (3) to (4);
            \draw (1) to (4);
        \end{tikzpicture}
    \end{center}
    \caption{A graph $G$ whose LCM lattice is not geometric but for with $\pd_S(S/I(G)) = \hgt(\L_{I(G)})$.}\label{fig:pdim=ht}
    \end{figure}
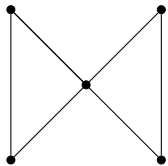
    By Theorem~\ref{GradedGraph}, $\L_{I(G)}$ is not graded and hence not upper or lower semimodular.
     However, one checks with Macaulay2 that  $\pd_S(S/I(G)) = \hgt(\L_{I(G)}) = 4$.
\end{example}

A lattice $L$ is called \term{strongly complemented} if for every $x \in L$, there exist elements $y,z \in L$ such that $y$ is a meet of coatoms, $z$ is a join of atoms, $x \vee y = x \vee z = \hat{1}$, and $x \wedge y = x \wedge z = \hat{0}$.  Clearly such lattices are complemented but not conversely.  

\begin{thm}\label{pdim:strongly:complemented:lattices}
Let $I \subseteq S$ be a monomial ideal.  If $\pd_S(S/I) = \hgt(\L_I)$, then $\L_I$ is strongly complemented.
\end{thm} 

\begin{proof}
    Suppose $\L_I$ is not strongly complemented.  By \cite[Theorem 3.3]{Bjorner81}, $O((\hat{0},\hat{1})_{\L_I})$ is contractible.  Hence $\tilde{H}_i(O((\hat{0},\hat{1})_{\L_I});\kk) = 0$ for all $i$.  By Theorem~\ref{LatticeHomology}, $\pd_S(S/I) < \hgt(\L_I)$.
\end{proof}

\begin{example}
    This example shows that the converse of the previous theorem does not hold.  Consider $I(P_5)$.  By Proposition~\ref{prop:special:graphs}, $\L_{I(P_5)}$ is complemented.  It is easy to check that it is also strongly complemented; however, $\pd_S(S/I(P_5)) = 3$, while $\hgt(\L_{I(P_5)}) = 4$.  
\end{example}

\begin{example}
    This example shows that there are LCM lattices of monomial ideals that are complemented but not strongly complemented.
    Consider the edge ideal of the graph pictured in Figure~\ref{fig:comp:not:strongly}.
    \begin{figure}[ht]
    \begin{center}
        \begin{tikzpicture}
            \node (0) at (0, 0) [circle,draw, fill,inner sep=0pt,minimum size=3pt] { };
            \node (1) at (.5,.5) [circle,draw, fill,inner sep=0pt,minimum size=3pt] { };
            \node (2) at (1,0) [circle,draw, fill,inner sep=0pt,minimum size=3pt] { };
            \node (3) at (.5, -.5) [circle,draw, fill,inner sep=0pt,minimum size=3pt] { };
            \node (4) at (1.5, 0) [circle,draw, fill,inner sep=0pt,minimum size=3pt] { };
            \node (5) at (-.5, 0) [circle,draw, fill,inner sep=0pt,minimum size=3pt] {};

            \draw (0) to (1);
            \draw (1) to (2);
            \draw (2) to (3);
            \draw (3) to (0);
            \draw (2) to (4);
            \draw (1) to (3);
            \draw (0) to (5);

            \node at (0,0.35) {2};
            \node at (-.5,0.35) {1};
            \node at (1,0.35) {5};
            \node at (1.5,0.35) {6};
            \node at (.5,0.85) {3};
            \node at (.5,-0.85) {4};

        \end{tikzpicture}
    \end{center}
    \caption{A graph $G$ whose LCM lattice complemented but not strongly complemented.}\label{fig:comp:not:strongly}
    \end{figure}
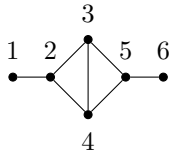
    It follows from Proposition~\ref{prop:complemented:LCM:lattice} that $\L_{I(G)}$ is complemented.  Coatoms correspond to induced subgraphs on $5$ of the $6$ vertices that don't have isolated vertices. Thus any element which is a meet of coatoms must include both $x_2$ and  $x_5$.  Any complement of $x_1x_2x_3x_4$ which is a meet of coatoms must be divisible by $x_1x_2, x_2x_3$ or $x_2x_4$ - a contradiction.  It follows that $\L_{I(G)}$ is not strongly complemented.  (In this case $\pd_S(S/I(G)) = 4$, while $\hgt(\L_{I(G)}) = 5$.)
\end{example}

\section{Characterizations of Lattice Properties of LCM Lattices of Edge Ideals}\label{section:edge:ideals}

It is likely intractable to classify the lattice properties of arbitrary monomial ideals in terms of the associated combinatorics of hypergraphs or Stanley-Reisner complexes.  Even checking whether a given monomial ideal is minimal is nontrivial.  However, for the case of squarefree monomial ideals generated in degree $2$, we provide a complete classification of many natural lattice properties.  The study of edge ideals with graded LCM lattices was previously done in work of Nevo and Peeva.  The remaining cases appear to be new.  We focus on each property in a separate subsection, grouping those that are equivalent.

First we note that it is sufficient to consider the case of connected graphs.  Indeed, if $G = G_1 \sqcup G_2$, then $\L_{I(G)} \cong \L_{I(G_1)} \times \L_{I(G_2)}$.  We are then reduced to the case of connected graphs by the following lemma, whose proof is an easy exercise.

\begin{lemma} Let $L_1,L_2$ be lattices and set $L = L_1 \times L_2$.  Then $L$ satisfies property (P) if and only if both of $L_1$ and $L_2$ do, where (P) is any of the following properties:
\begin{enumerate}
        \item Boolean
        \item distributive
        \item graded
        \item modular
        \item geometric
        \item upper semimodular
        \item lower semimodular
        \item atomic
        \item coatomic
        \item complemented
        \item supersolvable
    \end{enumerate}
\end{lemma}

\subsection{Graded LCM Lattices}

In this section we prove the converse of a 
 result of Nevo and Peeva characterizing edge ideals with graded LCM lattices.  Importantly, this result makes computing the rank of a monomial in a graded LCM lattice easy in the subsequent sections.

\begin{thm}\label{GradedGraph}
    Let $G$ be a simple, connected, nontrivial graph. Then the following are equivalent,
    \begin{enumerate}
        \item The complement graph $G^c$ is $C_4$-free.  (Sometimes also called gap-free.)
        \item $\L_{I(G)}$ is graded
        \item $I(G)$ is linearly presented.
    \end{enumerate}
    Moreover, in this case, if $1 \neq m \in \L_{I(G)}$, then $\rank(m) = \deg(m) - 1$.
\end{thm}

\begin{proof}
    That statements $\mathit{(1)}$ and $\mathit{(3)}$ are equivalent is proved in \cite[Proposition 1.3]{NP13}, where the authors attribute the result to a personal communication with Francisco, H\`a, and Van Tuyl and cite \cite[Theorem 3.2.4]{HVT07}.  That $\mathit{(1)}$ implies $\mathit{(2)}$ and the rank calculation is proved in \cite[Theorem 1.5]{NP13}.  

     For the converse, suppose $\L_{I(G)}$ is graded.  Since $G$ is connected, we can pick $v_1 \in V$ arbitrarily and then for $i > 1$ label $v_i$ so that it is adjacent to a vertex in $\{v_1,\ldots,v_{i-1}\}$.  Then $\L_{I(G)}$ contains a necessarily maximal chain of the form:
     \[1 < x_{v_1}x_{v_2} < x_{v_1}x_{v_2}x_{v_3} < \cdots < x_{v_1}\cdots x_{v_n},\]
     which has length $n-1$.  Thus $\rk(\L_{I(G)}) = n-1$.  Suppose $G$ is not gap-free, so that the induced subgraph on a vertex set $\{a,b,c,d\}$ is a pair of disconnected edges $\{a,b\}$ and $\{c,d\}$.  Once again, we pick vertices $w_i \in V$ for $5 \le i \le n$ so that $w_i$ is adjacent to a vertex in $\{a,b,c,d\} \cup \{w_5,\ldots,w_{i-1}\}$.  Then
     \[1 < x_ax_b < x_ax_bx_cx_d < x_ax_bx_cx_dx_{w_1} < \cdots < x_ax_bx_cx_dx_{w_1}x_{w_2}\cdots x_{w_{n-4}}\]
     is a saturated chain of length $n-2$, contradicting that $\L_{I(G)}$ was graded.  It follows that $G^c$ is $C_4$-free.
\end{proof}

\noindent It is notable that being linearly presented is a lattice property and does not depend on the characteristic of the coefficient field.

\subsection{Modular and Boolean LCM Lattices}

It turns out that very few edge ideals have Modular or Boolean lattices.

\begin{thm}\label{USS}
Let $G$ be a simple, connected, nontrivial graph.  The following are equivalent:
\begin{enumerate}
    \item  $\L_{I(G)}$ is modular
    \item $\L_{I(G)}$ is geometric
    \item $\L_{I(G)}$ is upper semimodular
    \item $G$ has no disjoint edges
    \item  $G$ is isomorphic to $C_3$ or $St_n$ for some $n \ge 2$.
\end{enumerate}
\end{thm}

\begin{proof}
    Since $\L_{I(G)}$ is always atomic, it is clear that \textit{(1)} $\Rightarrow$ \textit{(2)} $\Leftrightarrow$ \textit{(3)}.\\

   \noindent \underline{\textit{(3)} $\Rightarrow$ \textit{(4)}}: Suppose $\L_{I(G)}$ is upper semimodular.  In particular, it is graded, so by Theorem~\ref{GradedGraph}, for any $m \in \L_{I(G)}$, $\rk(m) = \deg(m) -1$.  Toward a contradiction, suppose there are vertices $a,b,c,d \in V$ such that $\{a,b\}, \{c,d\} \in E(G)$.  Setting $m_1 = x_ax_b$ and $m_2 = x_cx_d$, we have $m_1,m_2 \in \L_{I(G)}$ and yet
    \[\rk(m_1) + \rk(m_2) = \deg(m_1)-1 + \deg(m_2) -1 = 2 < 3 = \deg(m_1m_2) - 1 + \deg(1) = \rk(m_1 \wedge m_2) + \rk(m_1 \vee m_2),\]
    contradicting the upper semimodularity of $\L_{I(G)}$.\\
    
    \noindent \underline{\textit{(4)} $\Rightarrow$ \textit{(5)}}: If $G$ has no pair of disjoint edges, then either all edges are adjacent to a single vertex, in which case $G$ is a star graph, or else $G$ must be a 3-cycle.\\
    
    \noindent \underline{\textit{(5)} $\Rightarrow$ \textit{(1)}}: If $G = St_n$, then each monomial generator of $I(G)$ is divisible by a unique variable.  By Theorem~\ref{Boolean:LCM:lattice}, $\L_{I(G)}$ is Boolean, and hence modular.  If $G = C_3$, then $\rk(\L_{I(G)}) = 2$ and again $\L_{I(G)}$ is modular.
\end{proof}

The following characterization of edge ideals with Boolean LCM lattices is then an easy corollary.

\begin{cor}\label{BooleanEdgeIdeal}
    Let $G$ be a simple, connected, nontrivial graph. The following are equivalent
\begin{enumerate}
        \item $\L_{I(G)}$ is Boolean.
        \item $\L_{I(G)}$ is distributive.
        \item $G$ is a star graph.
    \end{enumerate}  
\end{cor}

\begin{proof} Since $\L_{I(G)}$ is atomic, $\L_{I(G)}$ is Boolean if and only if it is distributive by \cite[Proposition 3.4.5]{Stanley11}.  Since $\L_{I(C_3)}$ is not Boolean, the rest follows from Theorem~\ref{USS}.
\end{proof}

\subsection{Supersolvable LCM Lattices}

While it is quite restrictive for an edge ideal to have a modular LCM lattice, there are quit a few more that are merely supersolvable.

\begin{thm}\label{thm:supersolvable} Let $G$ be a simple, connected, nontrivial graph.  Then $\L_{I(G)}$ is supersolvable if and only if $G$ has an edge adjacent to every other edge.
\end{thm}

\begin{proof}
    Suppose no edge of $G$ is adjacent to every other edge. First, since supersolvable lattices are graded, we may assume that $\L_{I(G)}$ is graded. Take a maximal chain $C$ in $\L_{I(G)}$.  The first nontrivial element in the chain correspond to an edge $e = \{a,b\} \in E(G)$.  By assumption, there is another edge $\{c,d\} \in E(G)$ which is not adjacent to $e$.  Since $\L_{I(G)}$ is graded, $\rk(x_ax_b \vee x_cx_d) = \rk(x_ax_bx_cx_d) = \deg(x_ax_bx_cx_d) - 1 = 3$; yet $\rk(x_ax_b) + \rk(x_cx_d) = 2$, while $\rk(x_ax_b \wedge x_cx_d) + \rk(x_ax_b \vee x_cx_d) = 3 + 0$.  Thus no chain is modular and hence $\L_{I(G)}$ is not supersolvable.
   
    Now, suppose there is an edge $e = \{a,b\}$ adjacent to all other edges in $G$.  Relabel the vertices as follows: $V(G) \setminus e := \{1,\dotsc,m\}$.
    Then, consider the maximal chain 
    \begin{align*}
        M = \{1 \leq x_e \leq x_ex_1 \leq x_ex_1x_2 = x_ex_{[2]} \dotsb \leq x_ex_{[i]} \leq \dotsb \leq x_{e}x_{[m]}\}.
    \end{align*}
    We claim this maximal chain satisfies the condition required for supersolvability.
    Note that every edge in $G$ has $a$ or $b$ as an endpoint. Hence, $x_a$ or $x_b$ divides every atom of $\L_{I(G)}$.
    Then, every non-bottom element in $\L_{I(G)}$ is of the form $x_ax_S$ with $S \subset N(a)$, $x_bx_T$ with $T \subset N(b)$, or $x_ax_bx_R$ with $R \subset V(G) \setminus e$.
    By symmetry, we need only handle the first and third cases.
    
    Let $S \subset N(a)$ and $R \subset V(G) \setminus e$.
    We have:
    \begin{align*}
        x_ex_{[i]} \vee x_e x_R &= x_e x_{R \cup [i]},\\
        x_ex_{[i]} \wedge x_e x_R &= x_e x_{R\cap [i]},\\
        x_ex_{[i]} \vee x_a x_S &= x_e x_{S \cup [i]},\\
        x_ex_{[i]} \wedge x_a x_S &= \begin{cases} x_e x_{S\cap [i]} & \text{ if } S \cap [i] \neq \varnothing\\
        1 &\text{ otherwise.} \end{cases}\\
    \end{align*}
    Hence,
    \begin{align*}
        \operatorname{rk}(x_e x_{R \cup [i]}) + \operatorname{rk}(x_e x_{R\cap [i]}) &= (|R \cup [i]| + 2) - 1 + (|R \cap [i]| + 2) - 1 \\
        &= |R \cup [i]| + |R \cap [i]| + 2\\
        &= |R| + |[i]| + 2\\
        &= (|R| + 2 - 1) + (i + 2 -1) = \operatorname{rk}(x_e x_R) + \operatorname{rk}(x_e x_{[i]}).
    \end{align*}
    If $S \cap [i] \neq \varnothing$, then
    \begin{align*}
        \operatorname{rk}(x_e x_{S \cup [i]}) + \operatorname{rk}(x_a x_{S\cap [i]}) &= (|S \cup [i]| + 2) - 1 + (|S \cap [i]| + 1) - 1 \\
        &= |S \cup [i]| + |S \cap [i]| + 1\\
        &= |S| + |[i]| + 1\\
        &= (|S| + 1 - 1) + (i + 2 - 1) = \operatorname{rk}(x_a x_S) + \operatorname{rk}(x_e x_{[i]}).
    \end{align*}    
    If $S \cap [i] = \varnothing$, then
    \begin{align*}
        \operatorname{rk}(x_e x_{S \cup [i]}) + \operatorname{rk}(1) &= (|S \cup [i]| + 2) - 1 + 0 \\
        &= |S| + |[i]| + 1\\
        &= (|S| + 1 - 1) + (i + 2 - 1) = \operatorname{rk}(x_a x_S) + \operatorname{rk}(x_e x_{[i]}).
    \end{align*}
    It follows that $M$ is a modular chain and $\L_{I(G)}$ is supersolvable.
\end{proof}

By Theorem~\ref{thm:supersolvable}, graphs with supersolvable LCM lattice look like those pictured in Figure~\ref{figure:ss}.

    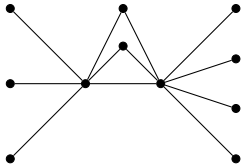
\begin{figure}[ht]
    \begin{tikzpicture}
        \node (a) at (-0.5, 0) [circle,draw, fill,inner sep=0pt,minimum size=3pt] {};
        \node (b) at (0.5, 0) [circle,draw, fill,inner sep=0pt,minimum size=3pt] {};
        \node (c) at (1.5, 1) [circle,draw, fill,inner sep=0pt,minimum size=3pt] {};
        \node (d) at (1.5, 0.33) [circle,draw, fill,inner sep=0pt,minimum size=3pt] {};
        \node (e) at (1.5, -0.33) [circle,draw, fill,inner sep=0pt,minimum size=3pt] {};
        \node (f) at (1.5, -1) [circle,draw, fill,inner sep=0pt,minimum size=3pt] {};
        \node (g) at (-1.5, 1) [circle,draw, fill,inner sep=0pt,minimum size=3pt] {};
        \node (h) at (-1.5, 0) [circle,draw, fill,inner sep=0pt,minimum size=3pt] {};
        \node (i) at (-1.5, -1) [circle,draw, fill,inner sep=0pt,minimum size=3pt] {};
        \node (j) at (0, 1) [circle,draw, fill,inner sep=0pt,minimum size=3pt] {};
        \node (k) at (0, 0.5) [circle,draw, fill,inner sep=0pt,minimum size=3pt] {};
        \draw (a) -- (b);
        \draw (b) -- (c);
        \draw (b) -- (d);
        \draw (b) -- (e);
        \draw (b) -- (f);
        \draw (a) -- (g);
        \draw (a) -- (h);
        \draw (a) -- (i);
        \draw (a) -- (j);
        \draw (a) -- (k);
        \draw (b) -- (j);
        \draw (b) -- (k);
    \end{tikzpicture}
    \caption{A typical graph with supersolvable LCM lattice}\label{figure:ss}
    \end{figure}

\subsection{Lower Semimodular LCM Lattices}

Graphs with lower semimodular LCM lattices can be characterized by forbidden induced subgraphs, which leads to a simple classification of such graphs.  Note that the diamond graph is the graph on $\{1,2,3,4\}$ which has all possible edges but one.

\begin{thm}\label{thm:lsm} Let $G$ be a simple, connected, nontrivial graph.  Then the following are equivalent:
\begin{enumerate}
    \item $G$ has a clique $H$ such that every vertex of $G$ not in $H$ is adjacent to a unique vertex of $H$.
    \item $G$ is gap-free, $C_4$-free, and diamond-free.
    \item $\L_{I(G)}$ is lower semimodular.
\end{enumerate}
\end{thm}

\begin{proof}
    \noindent \underline{\textit{(1)} $\Rightarrow$ \textit{(3)}}: Let $G$ be as in \textit{(1)}.   Note that by Theorem~\ref{GradedGraph}, $\L_{I(G)}$ is graded and if $m \in \L_{I(G)}$, then $\rk(m) = \deg(m) - 1$.  Recall that elements of $\L_{I(G)}$ correspond to induced subgraphs with no isolated vertices.  Let $m,m' \in \L_{I(G)}$ with corresponding induced subgraphs $H,H'$.  If $|H \cap H'| \le 1$, then $m \wedge m' = 1$, while $m \vee m' = \lcm(m,m')$ has degree at least $\deg(m) + \deg(m')-1$.  Thus, $\rk(m) + \rk(m') = \deg(m) - 1 + \deg(m') - 1 \le \rk(m \wedge m') + \rk(m \vee m')$.  If $|H \cap H'| \ge 2$, then the structure of $G$ implies that $m \wedge m' = \gcd(m,m')$.  In particular, $\rk(m) + \rk(m') = \rk(m \wedge m') + \rk(m \vee m')$.  It follows that $\L_{I(G)}$ is lower semimodular.\\

     \noindent \underline{\textit{(3)} $\Rightarrow$ \textit{(2)}}: If $H$ is an induced subgraph of $G$, then $\L_{I(H)}$ is a convex sublattice of $\L_{I(G)}$.  Thus, if $\L_{I(H)}$ is not lower semimodular, neither is $\L_{I(G)}$.  Since lower semimodular lattices are in particular graded, $G$ must be gap-free by Theorem~\ref{GradedGraph}.  

     Let $H$ be either the 4-cycle or the diamond graph so that $V(H) = \{1,2,3,4\}$ and $\{1,2\},\{2,3\},\{3,4\},\{1,4\} \in E(H)$, while $\{1,3\} \notin E(H)$.  Set $m = x_1x_2x_3$ and $m' = x_1x_3x_4$.  Then $\rk(m) = \rk(m') = 2$, $\rk(m \vee m') = \rk(x_1x_2x_3x_4) = 3$, and $\rk(m \wedge m') = \rk(1) = 0$.  Thus, $\rk(m) + \rk(m') = 4 > 0 + 3 = \rk(m \wedge m') + \rk(m \vee m')$, and so $\L_{I(H)}$ is not lower semimodular.\\

      \noindent \underline{\textit{(2)} $\Rightarrow$ \textit{(1)}}: Suppose $G$ is a graph and let $H,H'$ be distinct maximal cliques of size at least $3$.  If $|H \cap H'| \ge 2$, let $u,v \in H \cap H'$.  Since $H,H'$ are maximal and distinct, there are vertices $w \in V(H)$ and $w' \in V(H')$ that are not adjacent.  The induced subgraph on vertices $\{u,v,w,w'\}$ is either a 4-cycle or a diamond, contradicting our assumption.

      If $|H \cap H'| = 1$, write $u = V(H \cap H')$.  Since $H,H'$ have size at least $3$ and are maximal cliques, there are at least two vertices, $v_1,v_2 \in V(H)$ which are each not adjacent to at least one vertex of $H'$.  Suppose we can find $w_1 \in V(H')$ not adjacent to both $v_1,v_2$.  By the same argument there is another vertex $w_2 \in V(H')$ not adjacent to at least one vertex of $H$.  If $w_2$ is not adjacent to both $v_1$ and $v_2$, then the induced subgraph on vertex set $\{v_1,v_2,w_1,w_2\}$ is a gap - a contradiction.  Else $w_2$ is adjacent to at least one of $v_1,v_2$, say $v_1$.  Now the induced subgraph on vertex set $\{v_1,u,w_1,w_2\}$ is a diamond - a contradition.  We are then left with the case that every vertex $v \in V(H)$ is adjacent to all but one vertices of $V(H')$ and vice versa.  Picking $v_1,v_2 \in V(H) \smallsetminus V(H')$ and $w_1,w_2 \in V(H') \smallsetminus V(H)$, we see that the induced subgraph on $\{v_1,v_2,w_1,w_2\}$ is a 4-cycle - a contradiction.  

      We are left with the case $|H \cap H'| = 0$.  Since $G$ is connected, there is a path from each vertex of $H$ to that of $H'$.  Since $G$ is gap-free, there is a vertex $v \in V(H)$ adjacent to a vertex $w \in V(H')$.  If there are 2 vertices $v_1,v_2 \in V(H)$ not adjacent to any vertices of $V(H')$, then picking any two vertices $w_1,w_2 \in V(H')$ leaves an induced subgraph on $\{v_1,v_2,w_1,w_2\}$ that is a gap - a contradiction.  Thus, there are two vertices $v_1,v_2 \in V(H)$ each adjacent to vertices $w_1,w_2 \in V(H')$; moreover, we can choose $w_1 \neq w_2$, since otherwise $w_1=w_2 \in V(H)$.  The ordered vertex set $\{v_1,w_1,w_2,v_2\}$ then forms a 4-cycle.  By assumption, the induced subgraph on $\{v_1,v_2,w_1,w_2\}$  must be a complete graph and so $\{v_1,w_2\} \in E(G)$.  Since $H$ and $H'$ are distinct maximal cliques, there is a vertex $w_3 \in V(H')$ not adjacent to $v_1$.  Then the induced subgraph on $\{v_1,w_1,w_2,w_3\}$ is a diamond - a contradiction.

      We conclude that $G$ has at most one maximal clique $H$ of size at least $3$.  If it has no such cliques, $G$ must be a tree of diameter at most $3$, which is of the claimed form.  If $G$ has a single clique $H$ of size $m \ge 3$, suppose $v \in V(G) \smallsetminus V(H)$.  Suppose $v$ is adjacent to some vertex in $H$.  Since $v \notin V(H)$, there is a vertex $w \in V(H)$ not adjacent to $v$.  If $v$ is adjacent to at least two vertices, $u_1,u_2 \in V(H)$, then the induced subgrapho on $\{u_1,u_2,v,w\}$ is a diamond - a contradiction.  Thus, such a $v$ is adjacent to a unique vertex of $H$.
      
      If $v$ is not adjacent to $H$, then $G$ contains a gap formed from $v$, an adjacent vertex $w$ (necessarily not in $H$ by maximality), and two more vertices of $H$ not adjacent to $w$ by the above argument - again a contradiction.

      This concludes the proof.
\end{proof}

   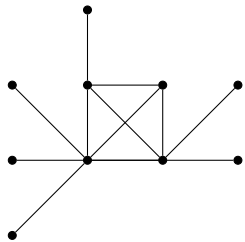
\begin{figure}[ht]
    \begin{tikzpicture}
        \node (a) at (0, 0) [circle,draw, fill,inner sep=0pt,minimum size=3pt] {};
        \node (b) at (1, 0) [circle,draw, fill,inner sep=0pt,minimum size=3pt] {};
        \node (c) at (1, 1) [circle,draw, fill,inner sep=0pt,minimum size=3pt] {};
        \node (d) at (0, 1) [circle,draw, fill,inner sep=0pt,minimum size=3pt] {};
        \node (e) at (0,2) [circle,draw, fill,inner sep=0pt,minimum size=3pt] {};
        \node (f) at (2,1) [circle,draw, fill,inner sep=0pt,minimum size=3pt] {};
        \node (g) at (2,0) [circle,draw, fill,inner sep=0pt,minimum size=3pt] {};
        \node (h) at (-1,0) [circle,draw, fill,inner sep=0pt,minimum size=3pt] {};
        \node (i) at (-1,1) [circle,draw, fill,inner sep=0pt,minimum size=3pt] {};
        \node (j) at (-1,-1) [circle,draw, fill,inner sep=0pt,minimum size=3pt] {};
    
        \draw (a) -- (b);
        \draw (b) -- (c);
        \draw (c) -- (d);
        \draw (a) -- (d);
        \draw (b) -- (f);
        \draw (a) -- (g);
        \draw (a) -- (h);
        \draw (a) -- (i);
        \draw (a) -- (j);
        \draw (a) -- (c);
        \draw (b) -- (d);
        \draw (d) -- (e);
    \end{tikzpicture}
    \caption{A typical graph with lower semimodular LCM lattice}\label{figure:lsm}
    \end{figure}

\subsection{Coatomic LCM Lattices}

While LCM lattices are always atomic, it is natural to ask when they are coatomic.  For edge ideals, there is a simple answer.

\begin{thm}\label{thm:coatomic:lcm:lattices}
    Let $G$ be a simple, connected, nontrivial graph.  Then $\L_{I(G)}$ is coatomic if and only $G = St_n$ for some $n$ or the every vertex has degree at least $2$.
\end{thm}

\begin{proof}
    Any Boolean lattice is coatomic, so we may assume that $G$ is not a star graph.  
    Set $n = |V(G)|$.  
    Note that elements of $\L_{I(G)}$ correspond to induced subgraphs with no isolated vertices.  Coatoms of $\L_{I(G)}$ all have degree $n-1$. 
    Therefore the coatoms correspond to induced subgraphs of $G$ on $n-1$ vertices with no isolated vertices.  The meet-irreducible elements, however, correspond to induced subgraphs on $n-1$ vertices after isolated vertices are discarded.  

    Suppose $v \in V$ and $\deg(v) = 1$ and let $w \in V$ be its unique neighbor.  Then the induced subgraph of $G$ on vertex set $V \smallsetminus\{v,w\}$ is a meet-irreducible element that is not a coatom, and thus not a meet of coatoms.

    On the other hand, if no vertices have degree $1$, then no induced subgraphs on $n-1$ vertices have isolated vertices.  Thus, all meet-irreducible elements are coatoms.  Since the meet-irreducible elements generate the lattice under the meet operation, $\L_{I(G)}$ is coatomic.
\end{proof}

\subsection{Complemented LCM Lattices}

While we give a characterization of complemented graphs, it is clearly not a hereditary condition and bit more difficult to visualize.

\begin{prop} \label{prop:complemented:LCM:lattice}
Let $G$ be a simple, connected, nontrivial graph. For any induced subgraph $H$ without isolated vertices, let $\mathcal{I}(H)$ denote the vertices $v \in V(G)$ connected only to vertices in $H$. The $\L_{I(G)}$ is complemented iff for every induced subgraph $H \in V$, there exists some independent set $X \subset H$ such that $\mathcal{I}(H) \subset N(X)$. 
\end{prop}

\begin{proof}
    Suppose $\L_{I(G)}$ is complemented. Then, let $x_A,x_B \in \L_{I(G)}$ be a pair of complements. Let $\mathcal{I}(A)$ denote the vertices that become isolated in $G[V\setminus A]$. Since $\mathcal{I}(A) \subset V \setminus A$, we have that $\mathcal{I}(A) \subset B$. Consider the set $B \cap A$. Since $B$ is a complement, $B \wedge A = \bigvee_{e \in B \cap A}e = \varnothing$. Hence, $B \cap A$ is an independent set. Note that since they correspond to elements of the LCM lattice, $A$ and $B$ are both unions of edges. As such, for each $a \in \mathcal{I}(A)$, there exists an edge $e_a \subset B$ to which it's incident. Then, $\bigcup_{a \in \mathcal{I}(A)}e_a \subset B$ satisfies $A\cap \bigcup_{a \in \mathcal{I}(A)}e_a \subset A \cap B = \varnothing$. We then have that $\mathcal{I}(A) \subset N(A\cap \bigcup_{a \in \mathcal{I}(A)}e_a)$.

    Suppose for every union of edges $A \subseteq V$, there exists some independent $X \subset A$ such that $\mathcal{I}(A) \subset N(X)$. Let $A \subseteq V$ be a union of edges. Choose any independent $X \subset A$ such that $\mathcal{I}(A) \subset N(X)$. We then have that for each $a \in \mathcal{I}(A)$, there exists an edge $e_a$ connecting some element in $X$ to $a$. Consider $B := \bigg(\bigcup_{a \in \mathcal{I}(A)}e_a\bigg)\cup (V(G)\setminus (A\cup \mathcal{I}(A)))$.
    Since every vertex in $V(G)\setminus (A\cup \mathcal{I}(A))$ is non-isolated in $G[V\setminus A]$, we have that $(V(G)\setminus (A\cup \mathcal{I}(A)))$ can be obtained as a union of edges. Hence $B$ is also a union of edges. Furthermore, $A \cap B = A\cap \bigcup_{a \in \mathcal{I}(A)}e_a \subset X$. Hence $A \wedge B = \varnothing$. At the same time, $A \cup B = A \cup (V(G)\setminus A) = V(G)$. 
    Hence $B$ is a complement for $A$.
\end{proof}

\subsection{LCM Lattices of Edge Ideals for Certain Graph Families}

In this section, we consider the lattice properties of LCM lattices of edge ideals for some common families of graphs.  A summary Venn diagram with example graphs of each type is given in Figure~\ref{venn}.  Most of their placements follows from the following Proposition~\ref{prop:special:graphs} and Propostion~\ref{prop:gray:areas}.  The rest are easily checked by hand.

\def\firstellip{(3.2, 0) ellipse [x radius=6cm, y radius=3cm, rotate=50]}
\def\secondellip{(0.6, 1cm) ellipse [x radius=6cm, y radius=3cm,
rotate=50]} \def\thirdellip{(-3.2, 0) ellipse [x radius=6cm, y radius=3cm,
rotate=-50]} \def\fourthellip{(-0.6, 1cm) ellipse [x radius=6cm, y
radius=3cm, rotate=-50]} 
 \def\fifthellip{(-1.6, 1cm) ellipse [x radius=8cm, y radius=5.5cm, rotate=-50]}
 \def\innercircle{(0, -1.2cm) circle [radius=.8cm]}

\begin{center}
\begin{figure}[!ht]
\resizebox{15cm}{!}{
\begin{tikzpicture} \filldraw[fill=black, opacity=0.2];


    \scope \fill[white] \secondellip; \fill[white] \thirdellip; \fill[white]
    \firstellip; \endscope

    \begin{scope}
        \begin{scope}[even odd rule]
            \clip \firstellip (-8,-8) rectangle (8,8);
        \fill[lightgray] \fourthellip;
        \end{scope}
    \end{scope}

     \begin{scope}
        \begin{scope}[even odd rule]
            \clip \firstellip (-8,-8) rectangle (8,8);
        \fill[lightgray] \thirdellip;
        \end{scope}
    \end{scope}

         \begin{scope}
        \begin{scope}[even odd rule]
            \clip \firstellip (-8,-8) rectangle (8,8);
            \clip \secondellip (-8,-8) rectangle (8,8);
        \fill[white] \thirdellip;
        \end{scope}
    \end{scope}

        \begin{scope}
        \begin{scope}[even odd rule]
            \clip \firstellip (-8,-8) rectangle (8,8);
            \clip \secondellip (-8,-8) rectangle (8,8);
        \fill[white] \fourthellip;
        \end{scope}
    \end{scope}
   
    \draw \firstellip node [label={[xshift=3.1cm, yshift=3.4cm]complemented}] {};
    \draw \secondellip node [label={[xshift=2.5cm, yshift=4.2cm]coatomic}] {};
    \draw \thirdellip node [label={[xshift=-3.0cm, yshift=4cm]LSM}] {};
    \draw \fourthellip node [label={[xshift=-2.4cm, yshift=4cm]supersolvable}] {};
    \draw \fifthellip node [label={[xshift=2.5cm, yshift=-7cm]graded}] {};
    \draw \innercircle node [label={Boolean}] {};
    \draw (0,-2.2cm) node  {modular} ;

        \node (a1) at (-2.5+0 + 3.4,-4.3+2.5 + 0.0-1.3) [circle, draw, fill, inner sep=1pt,minimum size = 3pt] {};
        \node (a2) at (-2.5+-0.3 + 3.5,-4.3+2.5 + -0.3-1.3) [draw, circle, fill, inner sep=1pt,minimum size = 3pt] {};
        \node (a3) at (-2.5+0.2 + 3.4,-4.3+2.5 + -0.3-1.3) [draw, fill, circle, inner sep=1pt,minimum size = 3pt] {};
        \node (a4) at (-2.5-0.65 + 3.5,-4.3+2.5 + -0.3-1.3) [draw, circle, fill, inner sep=1pt,minimum size = 3pt] {};
        \draw (a1) -- (a2);
        \draw (a2) -- (a3);
        \draw (a1) -- (a3);
        \draw (a4) -- (a2);

        \node (a1) at (-2+0.5/3 + 8.3,3.8-2) [circle, draw, fill, inner sep=1pt,minimum size = 3pt] {};
        \node (a2) at (-2+0.5 + 8.3,3.8-2) [circle, draw, fill, inner sep=1pt,minimum size = 3pt] {};
        \node (a3) at (-2+-0.5/3 + 8.3,3.8-2) [circle, draw, fill, inner sep=1pt,minimum size = 3pt] {};
        \node (a4) at (-2+-0.5 + 8.3,3.8-2) [circle, draw, fill, inner sep=1pt,minimum size = 3pt] {};
        \node (a5) at (-2+-.5 -1/3 + 8.3,3.8-2) [circle, draw, fill, inner sep=1pt,minimum size = 3pt] {};
        \draw (a1) -- (a2);
        \draw (a3) -- (a1);
        \draw (a4) -- (a3);
        \draw (a5) -- (a4);

        \node (a1) at (0,-1.5) [circle, draw, fill, inner sep=1pt,minimum size = 3pt] {};
        \node (a2) at (0,-1.5+.35) [circle, draw, fill, inner sep=1pt,minimum size = 3pt] {};
         \node (a3) at (.32,-1.5-.2) [circle, draw, fill, inner sep=1pt,minimum size = 3pt] {};
         \node (a4) at (-.32,-1.5-.2) [circle, draw, fill, inner sep=1pt,minimum size = 3pt] {};
         \draw (a1) -- (a2);
        \draw (a1) -- (a3);
        \draw (a1) -- (a4);

         \node (a1) at (0,-2.6) [circle, draw, fill, inner sep=1pt,minimum size = 3pt] {};
          \node (a2) at (.22,-3) [circle, draw, fill, inner sep=1pt,minimum size = 3pt] {};
           \node (a3) at (-.22,-3) [circle, draw, fill, inner sep=1pt,minimum size = 3pt] {};
         \draw (a1) -- (a2);
         \draw (a1) -- (a3);
         \draw (a2) -- (a3);

          \node (a1) at (-1.1,-3.1) [circle, draw, fill, inner sep=1pt,minimum size = 3pt] {};
           \node (a2) at (-.8,-3.4) [circle, draw, fill, inner sep=1pt,minimum size = 3pt] {}; 
           \node (a3) at (-.8,-3.1) [circle, draw, fill, inner sep=1pt,minimum size = 3pt] {}; 
           \node (a4) at (-1.1,-3.4) [circle, draw, fill, inner sep=1pt,minimum size = 3pt] {};
                    \draw (a1) -- (a2);
         \draw (a1) -- (a3);
         \draw (a1) -- (a4);
         \draw (a2) -- (a3);
         \draw (a2) -- (a4);
         \draw (a3) -- (a4);

\node (a1) at (0,-3.1-1) [circle, draw, fill, inner sep=1pt,minimum size = 3pt] {};
           \node (a2) at (.3,-3.4-1) [circle, draw, fill, inner sep=1pt,minimum size = 3pt] {}; 
           \node (a3) at (.3,-3.1-1) [circle, draw, fill, inner sep=1pt,minimum size = 3pt] {}; 
           \node (a4) at (0,-3.4-1) [circle, draw, fill, inner sep=1pt,minimum size = 3pt] {};
            \node (a5) at (-.3,-3.4-1) [circle, draw, fill, inner sep=1pt,minimum size = 3pt] {};
                    \draw (a1) -- (a2);
         \draw (a1) -- (a3);
         \draw (a1) -- (a4);
         \draw (a2) -- (a3);
         \draw (a2) -- (a4);
         \draw (a3) -- (a4);
                  \draw (a4) -- (a5);

\node (a1) at (-3,3) [circle, draw, fill, inner sep=1pt,minimum size = 3pt] {};
\node (a2) at (-3.3,3) [circle, draw, fill, inner sep=1pt,minimum size = 3pt] {};
\node (a3) at (-3.6,3) [circle, draw, fill, inner sep=1pt,minimum size = 3pt] {};
\node (a4) at (-3.9,3) [circle, draw, fill, inner sep=1pt,minimum size = 3pt] {};
                    \draw (a1) -- (a2);
                    \draw (a2) -- (a3);
                    \draw (a3) -- (a4);

\node (a1) at (-6.3,3) [circle, draw, fill, inner sep=1pt,minimum size = 3pt] {};
\node (a2) at (-6.3,2.6) [circle, draw, fill, inner sep=1pt,minimum size = 3pt] {};
\node (a3) at (-6.1,2.3) [circle, draw, fill, inner sep=1pt,minimum size = 3pt] {};
\node (a4) at (-6.5,2.3) [circle, draw, fill, inner sep=1pt,minimum size = 3pt] {};
\node (a5) at (-6.8,2.1) [circle, draw, fill, inner sep=1pt,minimum size = 3pt] {};
\node (a6) at (-5.8,2.1) [circle, draw, fill, inner sep=1pt,minimum size = 3pt] {};
                    \draw (a1) -- (a2);
                    \draw (a2) -- (a3);
                    \draw (a2) -- (a4);
                    \draw (a3) -- (a4);
                    \draw (a3) -- (a6);
                    \draw (a4) -- (a5);

\node (a1) at (-2,4.7) [circle, draw, fill, inner sep=1pt,minimum size = 3pt] {};
\node (a2) at (-2.4,4.7) [circle, draw, fill, inner sep=1pt,minimum size = 3pt] {};
\node (a3) at (-2,4.3) [circle, draw, fill, inner sep=1pt,minimum size = 3pt] {};
\node (a4) at (-2.4,4.3) [circle, draw, fill, inner sep=1pt,minimum size = 3pt] {};
\node (a5) at (-2.8,4.7) [circle, draw, fill, inner sep=1pt,minimum size = 3pt] {};
\node (a6) at (-1.6,4.3) [circle, draw, fill, inner sep=1pt,minimum size = 3pt] {};
                    \draw (a1) -- (a2);
                    \draw (a2) -- (a3);
                    \draw (a3) -- (a4);
                    \draw (a2) -- (a4);
                    
                    \draw (a1) -- (a3);
                    \draw (a2) -- (a5);
                    \draw (a4) -- (a6);

\node (a1) at (-2-2,4.7+2) [circle, draw, fill, inner sep=1pt,minimum size = 3pt] {};
\node (a2) at (-2.4-2,4.7+2) [circle, draw, fill, inner sep=1pt,minimum size = 3pt] {};
\node (a3) at (-2-2,4.3+2) [circle, draw, fill, inner sep=1pt,minimum size = 3pt] {};
\node (a4) at (-2-2.4,4.3+2) [circle, draw, fill, inner sep=1pt,minimum size = 3pt] {};
\node (a5) at (-2.8-2,4.7+2) [circle, draw, fill, inner sep=1pt,minimum size = 3pt] {};
                    \draw (a1) -- (a2);
                    \draw (a2) -- (a4);
                    
                    \draw (a1) -- (a3);
                    \draw (a2) -- (a5);
                    \draw (a4) -- (a3);

\node (a1) at (3,2-.2) [circle, draw, fill, inner sep=1pt,minimum size = 3pt] {};            \node (a2) at (3,2.4-.2) [circle, draw, fill, inner sep=1pt,minimum size = 3pt] {};  
\node (a3) at (3.4,2.2) [circle, draw, fill, inner sep=1pt,minimum size = 3pt] {};  
\node (a4) at (3.4,1.8) [circle, draw, fill, inner sep=1pt,minimum size = 3pt] {};  
\draw (a1) -- (a2);
\draw (a2) -- (a3);
\draw (a3) -- (a4);
\draw (a1) -- (a4);

\node (a1) at (3-2+.5,2-.2-2+.5) [circle, draw, fill, inner sep=1pt,minimum size = 3pt] {};            
\node (a2) at (3-2+.5,2.4-.2-2+.5) [circle, draw, fill, inner sep=1pt,minimum size = 3pt] {};  
\node (a3) at (3.4-2+.5,2.2-2+.5) [circle, draw, fill, inner sep=1pt,minimum size = 3pt] {};  
\node (a4) at (3.4-2+.5,1.8-2+.5) [circle, draw, fill, inner sep=1pt,minimum size = 3pt] {};  
\draw (a1) -- (a2);
\draw (a2) -- (a3);
\draw (a3) -- (a4);
\draw (a1) -- (a4);
\draw (a2) -- (a4);

\node (a1) at (3.9+.3,3.5+.3) [circle, draw, fill, inner sep=1pt,minimum size = 3pt] {}; 
\node (a2) at (4.3+.3,3.5+.3) [circle, draw, fill, inner sep=1pt,minimum size = 3pt] {};
\node (a3) at (4.5+.3,3.5-.35+.3) [circle, draw, fill, inner sep=1pt,minimum size = 3pt] {};
\node (a4) at (4.3+.3,3.5-.7+.3) [circle, draw, fill, inner sep=1pt,minimum size = 3pt] {};
\node (a5) at (3.9+.3,3.5-.7+.3) [circle, draw, fill, inner sep=1pt,minimum size = 3pt] {};
\node (a6) at (3.7+.3,3.5-.35+.3) [circle, draw, fill, inner sep=1pt,minimum size = 3pt] {};
\draw (a1) -- (a2);
\draw (a2) -- (a3);
\draw (a3) -- (a4);
\draw (a4) -- (a5);
\draw (a5) -- (a6);
\draw (a1) -- (a6);

\node (a1) at (2,-3) [circle, draw, fill, inner sep=1pt,minimum size = 3pt] {};            
\node (a2) at (2.4,-3) [circle, draw, fill, inner sep=1pt,minimum size = 3pt] {};  
\node (a3) at (2.4,-3.4) [circle, draw, fill, inner sep=1pt,minimum size = 3pt] {};  
\node (a4) at (2,-3.4) [circle, draw, fill, inner sep=1pt,minimum size = 3pt] {};  
\node (a5) at (2.8,-3.4) [circle, draw, fill, inner sep=1pt,minimum size = 3pt] {}; 
\draw (a1) -- (a2);
\draw (a2) -- (a3);
\draw (a3) -- (a4);
\draw (a1) -- (a4);
\draw (a1) -- (a3);
\draw (a3) -- (a5);

\node (a1) at (4,-1) [circle, draw, fill, inner sep=1pt,minimum size = 3pt] {};  
\node (a2) at (4.4,-1) [circle, draw, fill, inner sep=1pt,minimum size = 3pt] {};         
\node (a3) at (4.8,-1) [circle, draw, fill, inner sep=1pt,minimum size = 3pt] {};         
\node (a4) at (4.4,-1.4) [circle, draw, fill, inner sep=1pt,minimum size = 3pt] {};         \node (a5) at (4.2,-.65) [circle, draw, fill, inner sep=1pt,minimum size = 3pt] {};         
\node (a6) at (4.6,-.65) [circle, draw, fill, inner sep=1pt,minimum size = 3pt] {};         
\draw (a1) -- (a2);
\draw (a2) -- (a3);
\draw (a2) -- (a4);
\draw (a1) -- (a5);
\draw (a2) -- (a5);
\draw (a2) -- (a6);
\draw (a3) -- (a6);
\draw (a5) -- (a6);

\node (a1) at (3.9+.3-3,3.5+.3+1) [circle, draw, fill, inner sep=1pt,minimum size = 3pt] {}; 
\node (a2) at (4.3+.3-3,3.5+.3+1) [circle, draw, fill, inner sep=1pt,minimum size = 3pt] {};
\node (a3) at (4.5+.3-3,3.5-.35+.3+1) [circle, draw, fill, inner sep=1pt,minimum size = 3pt] {};
\node (a4) at (4.3+.3-3,3.5-.7+.3+1) [circle, draw, fill, inner sep=1pt,minimum size = 3pt] {};
\node (a5) at (3.9+.3-3,3.5-.7+.3+1) [circle, draw, fill, inner sep=1pt,minimum size = 3pt] {};
\node (a6) at (3.7+.3-3,3.5-.35+.3+1) [circle, draw, fill, inner sep=1pt,minimum size = 3pt] {};
\draw (a1) -- (a2);
\draw (a2) -- (a3);
\draw (a3) -- (a4);
\draw (a4) -- (a5);
\draw (a5) -- (a6);
\draw (a1) -- (a6);
\draw (a1) -- (a4);
\draw (a2) -- (a5);

\node (a1) at (3,5.2) [circle, draw, fill, inner sep=1pt,minimum size = 3pt] {}; 
\node (a2) at (3,4.8) [circle, draw, fill, inner sep=1pt,minimum size = 3pt] {}; 
\node (a3) at (3,4.4) [circle, draw, fill, inner sep=1pt,minimum size = 3pt] {}; 
\node (a4) at (3.4,5.2) [circle, draw, fill, inner sep=1pt,minimum size = 3pt] {}; 
\node (a5) at (3.4,4.8) [circle, draw, fill, inner sep=1pt,minimum size = 3pt] {}; 
\node (a6) at (3.4,4.4) [circle, draw, fill, inner sep=1pt,minimum size = 3pt] {}; 
\node (a7) at (3.2,5) [circle, draw, fill, inner sep=1pt,minimum size = 3pt] {}; 
\draw (a1) -- (a2);
\draw (a2) -- (a3);
\draw (a3) -- (a6);
\draw (a6) -- (a5);
\draw (a5) -- (a4);
\draw (a1) -- (a4);
\draw (a2) -- (a5);
\draw (a2) -- (a7);
\draw (a5) -- (a7);

\node (a1) at (0,-6.5) [circle, draw, fill, inner sep=1pt,minimum size = 3pt] {}; 
\node (a2) at (0.4,-6.5) [circle, draw, fill, inner sep=1pt,minimum size = 3pt] {}; 
\node (a3) at (0.8,-6.5) [circle, draw, fill, inner sep=1pt,minimum size = 3pt] {}; 
\node (a4) at (1.2,-6.5) [circle, draw, fill, inner sep=1pt,minimum size = 3pt] {}; 
\node (a5) at (-.4,-6.5) [circle, draw, fill, inner sep=1pt,minimum size = 3pt] {}; 
\node (a6) at (-.8,-6.5) [circle, draw, fill, inner sep=1pt,minimum size = 3pt] {}; 
\node (a7) at (-1.2,-6.5) [circle, draw, fill, inner sep=1pt,minimum size = 3pt] {}; 
\draw (a1) -- (a2);
\draw (a2) -- (a3);
\draw (a3) -- (a4);
\draw (a1) -- (a5);
\draw (a5) -- (a6);
\draw (a6) -- (a7);

\end{tikzpicture} 
}
\caption{Venn Diagram for LCM lattices of edge ideals}\label{venn}
\end{figure}
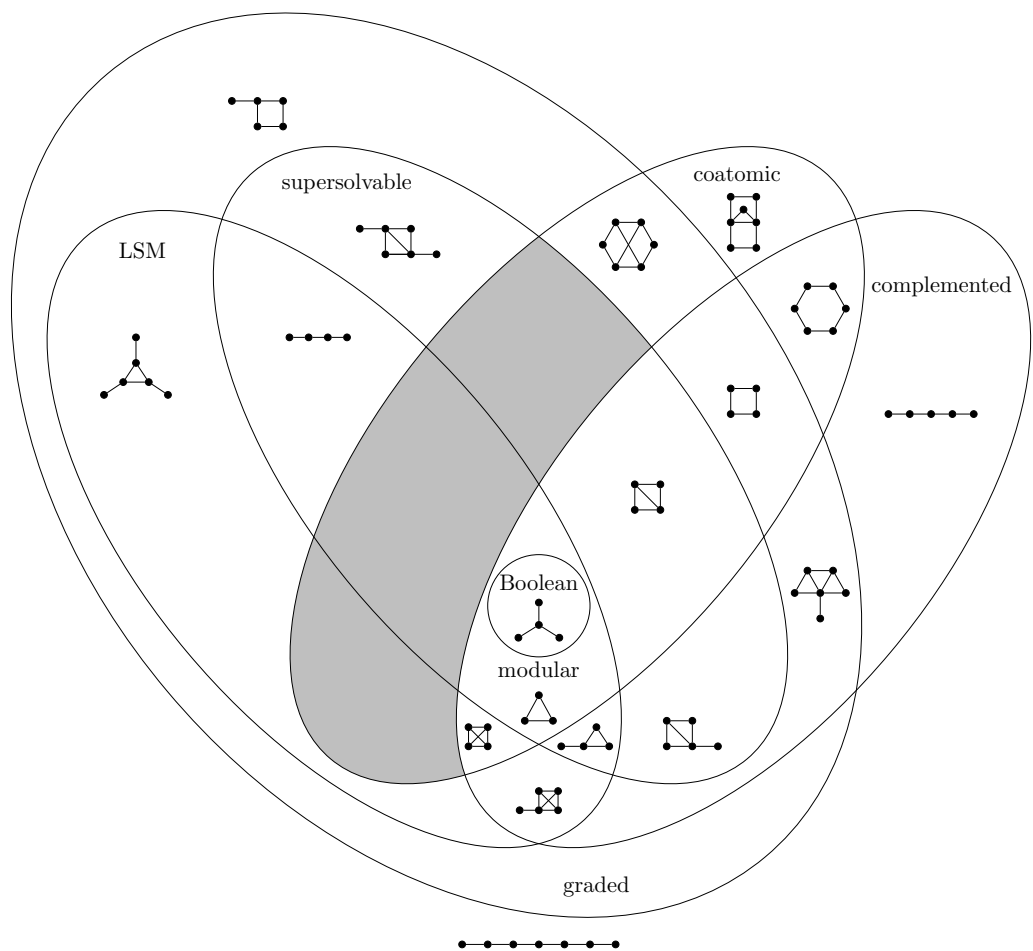
\end{center}

\begin{prop}\label{prop:special:graphs}
    We have the following characterizations:    
\begin{enumerate}
    \item The path graph $P_n$ has a \begin{itemize}
        \item[$\bullet$] graded LCM lattice iff $n \leq 4$.
        \item[$\bullet$] complemented LCM lattice iff $n \not\equiv 1 (\mathrm{mod } \, 3)$.
    \end{itemize}
    \item The cycle graph $C_n$ has a 
    \begin{itemize}
        \item[$\bullet$] graded LCM lattice iff $n \leq 5$.
        \item[$\bullet$] complemented LCM lattice for all $n \ge 3$.
    \end{itemize}
    \item The complete graph $K_n$ has a graded and complemented LCM lattice for all $n \geq 2$.
\end{enumerate}
\end{prop}

\begin{proof}
    \textit{(1)}: 
    For any $n \geq 5$, $P_n$ has an induced $P_5$. Thus, $P_n$ is gap-free if and only if $n \le 4$, which happens if and only if   $\L_{I(P_n)}$ is  graded by Theorem~\ref{GradedGraph}.

    Now suppose $n \equiv 1 \,(\mathrm{mod}\,3)$ with $n > 1$. We show that $\L_{I(P_n)}$ is not complemented. To this end, consider the element indexed by the set $Y = \{i \in [n] : i \equiv 0,2 \,(\mathrm{mod}\,3)\}$ and $T = [n] \smallsetminus Y$. Note that $x_Y \in \L_{I(P_n)}$.  We claim no complement exists for $x_Y$.  Toward a contradiction, suppose $x_C$ is a complement.  Since $C$ is a union of edges of $P_n$, for all $i \in [n]$ with $i \equiv 1 \,(\mathrm{mod} \, 3)$, $i \in C$ and either $i+1 \in C$ or $i-1 \in C$.  Since $0 \notin [n]$, $2 \in C$, and since $n+1 \notin [n]$, $n-1 \in C$.  Thus, there is a first index $1 < j < n$ with $j \equiv 0 \,(\mathrm{mod} \, 3)$ and $j \in C$.  Therefore $j-3 \notin C$ and so $j-1 \in C$.  It follows that $x_{j-1}x_{j} \mid (x_Y \wedge x_C)$ and so $x_C$ is not a complement for $x_Y$. 
    
    Suppose $n \not\equiv 1\,(\mathrm{mod} \,3)$.  Let $Y \subseteq [n]$ so that $x_Y \in \L_{I(P_n)}$, meaning that $Y$ is a union of edges of $P_n$.  Set $X = [n] \smallsetminus Y$.  As $X$ records the gaps not covered by $Y$, we can write $X = \bigcup_{i = 1}^t [a_i,b_i]$, where $1 \le a_1 \le b_1 < a_2 \le b_2 < \cdots < a_t \le b_t \le n$, $a_{i+1} \ge b_i+2$ for $i < n$, and $[a_i,b_i] = \{x \in \mathbb{Z} \mid a_i \le x \le b_i\}$.
    We consider the following cases:\\
    \noindent \underline{Case 1: $a_1 > 1$}: Take $\tilde{X} = \bigcup_{i = 1}^t [a_i-1,b_i]$.  Observe that $x_{\tilde{X}} \in \L_{I(P_n)}$, $x_{\tilde{X}} \wedge x_Y = x_{[n]}$, and $x_{\tilde{X}} \wedge x_Y = 1$.\\
    \noindent \underline{Case 2: $b_t < n$}: Take $\tilde{X} = \bigcup_{i = 1}^t [a_i,b_i+1]$.  Once again $x_{\tilde{X}} \in \L_{I(P_n)}$, $x_{\tilde{X}} \wedge x_Y = x_{[n]}$, and $x_{\tilde{X}} \wedge x_Y = 1$.\\
    \noindent \underline{Case 3: There exists $1 \le j \le n$ with $a_{j+1} > b_j + 2$}: In this case we choose $\tilde{X} = \bigcup_{i = 1}^j [a_i,b_i+1] \cup \bigcup_{i = j+1}^t [a_{i}-1,b_i]$, and the same conclusion follows.\\
    \noindent \underline{Case 4: There exists $1 \le j \le n$ with $b_j > a_j$}:
    Now we choose $\tilde{X} = \bigcup_{i = 1}^{j-1} [a_i,b_i+1] \cup [a_j,b_j] \cup \bigcup_{i = j+1}^t [a_{i}-1,b_i]$.\\
    \noindent \underline{Case 5: None of Cases 1-4 hold}: If none of these cases hold, it follows that $a_i = b_i$ for all $i$, $a_{i+1} = b_i+2$ for all $i < n$, $a_1 = 1$, and $b_t = n$.  It follows by induction that $n \equiv 1\,(\mathrm{mod} \,3)$ - contradiction.\\
    
    \noindent \textit{(2)}:
    For $n \geq 6$, $C_n$ has an induced $P_5$ and is therefore not graded. For $n = 3,4,5$, the graph of $C_n$ is gap-free and thus graded by Theorem~\ref{GradedGraph}. 

    Suppose $n \geq 3$. Label the vertices of $C_n$ by the elements $V := \{\bar{0},\bar{1},\dotsc,\bar{n-1}\} = \mathbb{Z}/n\mathbb{Z}$ such that each $\bar{i}$ is adjacent to $\bar{i+1}$.
    Now, take a union of edges $Y$ so that $x_Y \in \L_{I(C_n)}$. Then consider $T = \bigcup_{\bar{i} \notin Y} \{\bar{i},\bar{i+1}\}$.
    We have $T \cup Y = V$ since the complement of $Y$ is contained in $T$.
    Observe that no edges in $S$ are induced in $T$ since for any $\bar{i} \notin Y$, we must have $\bar{i+1},\bar{i+2} \in Y$.  It follows that $x_T \vee x_Y = x_{[n]}$ and $x_T \wedge x_Y = 1$.\\

    \noindent \textit{(3)}:
    That $\L_{I(K_n)}$ is graded follows easily from Theorem~\ref{GradedGraph}. 

    Let the vertices of $K_n$ be denoted by $[n]$. To see that it is also complemented, suppose $Y \subsetneq [n]$.  If $Y = [n] - \{i\}$, then taking $T = \{1,i\}$ we have $x_Y \vee x_T = x_{[n]}$ and $x_Y \wedge x_T = 1$. Otherwise, take $T = [n] - Y$ and the same conclusion follows. Thus, $\L_{I(K_n)}$ is complemented.
\end{proof}

The Venn diagram in Figure~\ref{venn} reveals a few interesting features not true of more general lattices.  The gray areas indicate that there are no such graphs with the indicated properties. 

\begin{prop}\label{prop:gray:areas}
    Let $G$ be a simple, connected, nontrivial graph.  
    \begin{enumerate}
        \item If $\L_{I(G)}$ is supersolvable and coatomic, then it is complemented.
        \item If $\L_{I(G)}$ is lower semimodular and coatomic, then it is complemented.
        \item If $\L_{I(G)}$ is supersolvable, lower semimodular, and coatomic, then it is modular.
    \end{enumerate}
\end{prop}

\begin{proof}
\noindent \textit{(1)}: Suppose $\L_{I(G)}$ is supersolvable and coatomic.  If $G$ is a star graph, then $\L_{I(G)}$ is Boolean and hence complemented; so we may assume that $G$ is not a star graph.  By Theorem~\ref{thm:supersolvable}, $G$ has an edge $\{v,w\}$ that is adjacent to every other edge.  By Theorem~\ref{thm:coatomic:lcm:lattices}, there are no edges of degree one.  So we can take $V(G) = \{v,w,u_1,\ldots,u_n\}$ for some integer $n$ and $E(G) = \{\{v,w\}\} \cup \bigcup_{i = 1}^n\{\{v,x_i\},\{w,x_i\}\}$.  Thus, every non-bottom element of $\L_{I(G)}$ has the form $v \prod_{i \in A} x_i$, $w \prod_{i \in A} x_i$ or $vw \prod_{i \in A} x_i$, where $A \subseteq [n]$ is arbitrary.  For a monomial of the first type, $w \prod_{i \notin A} x_i$ serves as a complement element.  For a monomial of the latter two types,  $v \prod_{i \notin A} x_i$ is a complement.  It follows that $\L_{I(G)}$ is a complemented lattice.\\

\noindent \textit{(2)}: Suppose $\L_{I(G)}$ is lower semimodular and coatomic.  If $G$ is a star graph, then $\L_{I(G)}$ is Boolean and hence complemented; so again we may assume that $G$ is not a star graph.  By Theorem~\ref{thm:coatomic:lcm:lattices}, $G$ has no degree one vertices.  Thus, by Theorem~\ref{thm:lsm}, $G$ must be a complete graph.  By Proposition~\ref{prop:special:graphs}, $\L_{I(G)}$ is complemented.\\

\noindent \textit{(3)}: Suppose $\L_{I(G)}$ is supersolvable, lower semimodular, and coatomic.  By the parts \textit{(1)} and \textit{(2)}, $G$ is either a star graph, or it is a complete graph in which one edge is adjacent to every other edge.  Thus, $G = K_3$ and $\L_{I(G)}$ is modular by Theorem~\ref{USS}.
\end{proof}

\subsection{Free Resolutions and Regularity}\label{freeres}

In this section we observe that certain lattice properties of LCM lattices of edge ideals yield strong conclusions regarding their graded free resolutions.  Recall that the \term{regularity} of a graded $S$-module is $\reg(M) = \max\{j-i \mid \beta_{ij}(M) \neq 0 \text{ for some } i\}$.  It is not hard to see that the regularity of the edge ideal of a connected graph $G$ can be arbitrarily large when $\L_{I(G)}$ is either coatomic or complemented.  Indeed, consider the $n$-cycle $C_n$.  By Theorem~\ref{thm:coatomic:lcm:lattices} and Proposition~\ref{prop:special:graphs}, $\L(I(C_n))$ is both coatomic and complemented for all $n \ge 3$.  However, the regularity of the corresponding edge ideals is unbounded, since $\reg(S/I(C_n)) = \lfloor \frac{n+1}{3} \rfloor$ for all $n \ge 3$.

If $\L(I(G))$ is graded, then $I(G)$ is linearly presented by Theorem~\ref{GradedGraph}, meaning $\beta_{2,4}(S/I(G)) = 0$.  These correspond to the gap-free graphs.  It is interesting that even for this class, the regularity can be arbitrarily large; see \cite{CKV16}.  Below we observe that for edge ideals with either supersolvable or lower semimodular LCM lattices, the corresponding resolutions are linear, that is, $\reg(S/I(G)) = 1$.  

\begin{thm} Let $G$ be a simple, connected, nontrivial graph.  If $\L(I(G))$ is supersolvable or lower semimodular, then $I(G)$ has a linear free resolution.  
\end{thm}

\begin{proof}
    By a well-known theorem of Fr\"oberg, $I(G)$ has a linear free resolution if and only if the complement graph $\bar{G}$ is chordal, see e.g. \cite[Theorem 9.2.3]{Herzog11}.  If $\L(I(G))$ is supersolvable, $G$ has a single edge adjacent to all other edges by Theorem~\ref{thm:supersolvable}.  Write $V(G) = \{v_1,\ldots,v_n\}$, where the edge $\{v_1,v_2\}$ is adjacent to every other edge.  Then $\bar{G}$ contains a complete subgraph on $\{v_3,v_4,\ldots,v_n\}$.  The additional vertices $v_1$ and $v_2$ cone over those vertices they are not connected to in $G$.  Any cycle in $G$ of length at least 4 then has at least 2 nonconsecutive vertices in $\{v_3,\ldots,v_n\}$ and thus has a chord.

    Now suppose that $\L(I(G))$ is lower semimodular.  By Theorem~\ref{thm:lsm}, $G$ contains a clique $H$ such that every vertex not in $H$ is adjacent to a unique vertex of $H$.  Write $V(G) = \{v_1,\ldots,v_n\} \sqcup \{w_1,\ldots,w_m\}$, where $V(H) = \{v_1,\ldots,v_n\}$.  Thus, $\bar{G}$ contains a clique on vertices $\{w_1,\ldots,w_m\}$.  Each of the vertices $v_1,\ldots,v_n$ is connected to a subset of the vertices in the clique and not to each other.  Once again, any cycle of length at least 4 has nonconsecutive vertices in $\{w_1,\ldots,w_m\}$ and thus has a chord.
\end{proof}

  It follows from \cite[Theorem 10.2.6]{Herzog11} that all powers of $I(G)$ also have linear free resolutions in this case.  In particular, edge ideals of connected graphs with geometric or modular LCM lattices have linear free resolutions and linear powers as well.

\section{LCM Lattices of Gorenstein Monomial Ideals}\label{sec:Gorenstein}

The LCM lattice of a monomial complete intersection is Boolean.  It is natural to study necessary and sufficient LCM lattice conditions for Gorenstein monomial ideals.  Given that the modular lattices, such as those in Example~\ref{Fano-Plane}, are Cohen-Macaulay but rarely Gorenstein, there seem to be few sufficient conditions on the LCM lattice of a monomial ideal to force the Gorenstein condition.

It is also easy to see that assuming $S/I$ is Cohen-Macaulay does not readily guarantee nice lattice properties of $\mathcal{L}_I$.  Indeed, if $G = P_4$, then $S/I(G)$ is Cohen-Macaulay, but $\mathcal{L}_{I(G)}$ is neither coatomic nor complemented; this follows easily from Theorem~\ref{thm:coatomic:lcm:lattices} and Theorem~\ref{prop:special:graphs}.  The following example shows that even if $S/I(G)$ is Gorenstein, $\mathcal{L}_{I(G)}$ need not even be graded.

\begin{example}
        Consider the graph $G$ whose edge set defines a triangulation of a cube:
    \[E(G) = \{12,23,34,14,56,67,78,58,15,26,37,48,13,25,27,45,47,68\}.\]
    Above, $12$ is shorthand for $\{1,2\}$.  Then $I(\bar{G})$ is a Gorenstein edge ideal.  It is easy to check that $\bar{G}$ is connected and that $2547$ is an induced $4$-cycle in $
    G$.  Therefore, $\mathcal{L}_{I(\bar{G})}$ is not graded by Theorem~\ref{GradedGraph}.
\end{example}

\noindent It follows that LCM lattices of Gorenstein monomial ideals need not be supersolvable, upper or lower semimodular, modular, distributive, or Boolean.  However, unlike the Cohen-Macaulay setting, it is plausible that they are always coatomic and complemented.  We record this in the following question.

\begin{question} Let $I$ be a monomial ideal in a polynomial ring $S$ such that $S/I$ is Gorenstein.   Is $\mathcal{L}_{I}$ coatomic?  Is $\mathcal{L}_{I}$ complemented?
\end{question}

\noindent Initial computations indicate the answer may be positive. 
To further support the question, we use the results from the previous section to show that Gorenstein edge ideals have coatomic LCM lattices.

First we recall the basics of Stanley-Reisner theory.  Let $I \subseteq \kk[x_1,\ldots,x_n]$ be a squarefree monomial ideal.  The Stanley-Reisner simplicial complex associated to $I$ is
\[\Delta_I = \{F \subseteq [n] \mid \underline{x}^F \notin I\}.\]
Here $\underline{x}^F = \prod_{j \in F} x_j$.  Conversely, given a simplicial complex $\Delta$ on ground set $[n]$, the associated squarefree monomial ideal is
\[I_\Delta = (\underline{x}^F \mid F \notin \Delta)\subseteq \kk[x_1,\ldots,x_n].\]
Given a simplicial complex $\Delta$ on  $[n]$ and a subset $H \subseteq [n]$, the restriction of $\Delta$ to $H$ is $\Delta |_H = \{H \cap F \mid F \in \Delta\}$.  The (multi)graded Betti numbers of a squarefree monomial ideal $I$ can be computed via Hochster's formula in terms of the cohomology of restrictions of the Stanley-Reisner complex: 
\[\beta_{i,m}^S(S/I) = \dim_\kk \tilde{H}^{|m|-i-1}(\Delta_I|_m,\kk).\]
Here, we identify a subset $m \subseteq [n]$ with the squarefree multidegree that is $1$ in position $i$ when $i \in m$ and $0$ otherwise.  See \cite[Corollary 7.13]{FMS14}.

\begin{thm} Let $G$ be a simple, connected, nontrivial graph.  If $I(G) \subseteq S = \kk[x_1,\ldots,x_n]$ is Gorenstein, then $\mathcal{L}_{I(G)}$ is coatomic.
\end{thm}

\begin{proof}
    Suppose $I = I(G)$ is Gorenstein.  Fix a variable $x_i$ and consider the graded short exact sequence
    \[0 \to S/(I:x_i)(-1) \xrightarrow{x_i} S/I \to S/(I + (x_i)) \to 0.\]
    Let $c = \pd(S/I) = \hgt(I)$.  By \cite[Theorem 5.1]{Stanley96}, we may assume that $\beta^S_{c,n}(S/I) = 1$.  Since $S/I$ is Gorenstein and since $I$ contains no linear forms, $\beta_{c-1,n-1}(S/I) = 0$.  By \cite[Lemma 1.1]{Lyubeznik88}, $\pd(S/(I:x_i)) = c$ so that $S/(I:x_i)$ is  Cohen-Macaulay.  (See also \cite[Lemma 5.1]{DHS13}.) We claim that $S/(I:x_i)$ is Gorenstein.  Note that by the long exact sequence for $\Tor$, $\pd(S/(I+(x_i)) \le c+1$.  Write $\Delta' = \Delta|_{[n]\smallsetminus i}$.  Then $\Delta_{I + (x_i)} = \Delta'$.
    For any squarefree multidegree $m$, by above we have
\begin{align*}
\beta_{c+1,m}(S/(I + (x_i))) &= \dim_\kk \tilde{H}^{|m| - c}(\Delta'|_m,\kk)\\
&= \begin{cases}
    \dim_\kk \tilde{H}^{|m| - c }(\Delta|_m,\kk) = \beta_{c+1,m}(S/I) = 0 & \text{ if } i \notin m\\
    \dim_\kk \tilde{H}^{|m| - c}(\Delta|_{m \smallsetminus i},\kk) = \beta_{c,m\smallsetminus i}(S/I) = 0 & \text{ if } i \in m.\\
\end{cases}
    \end{align*}
Thus $\pd(S/(I + (x_i))) = c$.  By the long exact sequence of $\Tor$ we have that for every integer $d$,
\[0 = \Tor_{c+1}(S/(I + (x_i)),\kk)_d \to \Tor_c(S/(I:x_i),\kk)_{d-1} \to \Tor_c(S/I,\kk)_d\]
is exact.  Since $\dim_\kk \Tor_c(S/I,\kk)_d = \beta_{c,d}(S/I) = 0$ for $d \neq n$, and $\beta_{c,n}(S/I) = 1$, it follows that $$\beta_{c,n-1}(S/(I:x_i)) = \beta_{c}(S/(I:x_i)) = 1;$$ in particular, $S/(I:x_i)$ is Gorenstein.

 Now, toward a contradiction, suppose that $\mathcal{L}_{I(G)}$ is not coatomic.  Then, by Theorem~\ref{thm:coatomic:lcm:lattices}, $G$ is not a star graph and $G$ has a vertex $i$ of degree $1$. Thus we may assume that the unique neighbor $j$ of the vertex $i$ has at least one additional neighboring vertex.  Since vertex $i$ has a unique neighbor $j$, $I:x_i = I + (x_j)$, which is Gorenstein by above.  Thus we have a graded short exact sequence
 \[0 \to S/(I:x_j)(-1) \to S/I \to S/(I + (x_j)) \to 0.\]
Once again $\pd(S/I) = \pd(S/(I:x_j)) = \pd(S/(I + (x_j))) = c$.  Since $S/(I + (x_j))$ is Gorenstein, $\beta_c(S/(I+(x_j))) = 1$.  Since $S/I$ is Gorenstein and since $I$ contains no linear forms, $\beta_{c,n}(S/I) = 1$ and $\beta_{c-1,n-1}(S/I) = 0$.  By the long exact sequence of $\Tor$, we have
\[\Tor_c(S/(I+(x_j),\kk)_{n-1} \to \Tor_{c-1}(S/(I:x_j),\kk)_{n-2} \to \Tor_{c-1}(S/I,\kk)_{n-1} = 0\]
is exact.  Since $I:x_j$ contains at least 2 variables, the free resolution of $S/(I:x_j))$ is a Koszul complex on those variables tensored with the free resolution of the nondegenerate part of the ideal.  Therefore,
$$\dim_\kk \Tor_{c-1}(S/(I:x_j),\kk)_{n-2} = \beta_{c-1,n-2}(S/(I:x_j)) \ge 2 \cdot \beta_{c,n-1}(S/(I:x_j)) = 2,$$ while 
$$\dim_\kk \Tor_c(S/(I+(x_j)),\kk)_{n-1} = \beta_{c,n-1}(S/(I + (x_j)) \le \beta_{c}(S/(I + (x_j)) = 1,$$ which contradicts the exactness above.  It follows that $\mathcal{L}_{I(G)}$ must be coatomic.
\end{proof}

It is not clear how to extend the previous result to arbitrary Gorenstein monomial ideals.  That their LCM lattices should also be complemented is natural, given that any element of the LCM lattice that supports nonzero homology has a complement by the duality of the free resolution; however, there are often elements of the LCM lattice that do not support homology and thus do not correspond to nonzero graded Betti numbers.  Finding complements of those elements would be key to proving such lattices are complemented.

\section*{Acknowledgements} The authors thank Jonas Hartwig and Ryan Martin for valuable feedback, and Vinh Nguyen for suggesting the connection with matroid cover ideals, leading to  Theorem~\ref{cover}.  The authors also thank the anonymous referees whose comments notably led to the addition of subsection~\ref{freeres} and section~\ref{sec:Gorenstein}. Both authors were partially supported by National Science Foundation grant  DMS--2401256.

\end{spacing}

\bibliographystyle{alphaurl}
\bibliography{lcm}

\end{document}